\documentclass[11pt]{amsart}

\protected\def\ignorethis#1\endignorethis{}
\let\endignorethis\relax

\def\TOCstop{\addtocontents{toc}{\ignorethis}}
\def\TOCstart{\addtocontents{toc}{\endignorethis}}

\usepackage[top=3.0cm, bottom=3.0cm, inner=3.0cm, outer=3.0cm, includefoot]{geometry}

\usepackage[latin1]{inputenc}
\usepackage[T1]{fontenc}
\usepackage{verbatim}

\usepackage{geometry}
\usepackage{amssymb}
\usepackage{amsmath}
\usepackage{graphicx}
\usepackage{amsthm}

\usepackage{color}

\usepackage{enumerate}

\newcommand{\IR}{{\mathbb{R}}}

\newcommand{\IN}{{\mathbb{N}_0}}
\newcommand{\IZ}{{\mathbb{Z}}}

\newcommand{\eChar}{\begin{enumerate}[(i)]}
\newcommand{\eCharR}{\begin{enumerate}[(a)]}
\newcommand{\eBr}{\begin{enumerate}[(1)]}

\newcommand{\vol}{\operatorname{vol}}

\newcommand{\eps}{\varepsilon}

\newcommand{\Abstract}

\newcommand{\Gc}{\mathcal{G}}
\newcommand{\Fc}{\mathcal{F}}

\newcommand{\R}{{\mathbb R}}
\newcommand{\N}{{\mathbb N}}
\newcommand{\Z}{{\mathbb Z}}

\newcommand{\CDE}{\mathrm{CDE}}
\newcommand{\CCD}{\mathrm{CD}}

\newcommand{\Hm}[1]{\leavevmode{\marginpar{\tiny%
$\hbox to 0mm{\hspace*{-0.5mm}$\leftarrow$\hss}%
\vcenter{\vrule depth 0.1mm height 0.1mm width \the\marginparwidth}%
\hbox to 0mm{\hss$\rightarrow$\hspace*{-0.5mm}}$\\\relax\raggedright
#1}}}

\title[Ricci curvature on birth-death processes]{Ricci curvature on birth-death processes}

\author{Bobo Hua}

\author{Florentin Münch}

\date{\today}

\theoremstyle{plain}
\newtheorem{lemma}{Lemma}[section]
\newtheorem{theorem}[lemma]{Theorem}
\newtheorem{proposition}[lemma]{Proposition}
\newtheorem{corollary}[lemma]{Corollary}

\theoremstyle{definition}

\newtheorem{example}[lemma]{Example}

\newtheorem{rem}[lemma]{Remark}
\newtheorem{defn}[lemma]{Definition}

\numberwithin{equation}{section}

\begin{document}








\begin{abstract}
In this paper, we study curvature dimension conditions on  birth-death processes which correspond to linear graphs, i.e., 
weighted graphs supported on the infinite line or the half line. We give a combinatorial characterization of Bakry and Émery's $\CCD(K,n)$ condition for linear graphs and prove the triviality of edge weights for every linear graph supported on the infinite line $\IZ$ with non-negative curvature.
Moreover, we show that linear graphs with curvature decaying not faster than $-R^2$ are stochastically complete.
We deduce a type of Bishop-Gromov comparison theorem for normalized linear graphs. For normalized linear graphs with non-negative curvature, we obtain the volume doubling property and
the Poincar\'e inequality, which yield Gaussian heat kernel estimates and parabolic Harnack inequality by Delmotte's result. As applications, we generalize the volume growth and stochastic completeness properties to weakly spherically symmetric graphs. Furthermore, we give examples of infinite graphs with a positive lower curvature bound. 
\end{abstract}

\maketitle

\tableofcontents


\section{Introduction}\label{s:intro}

One of the driving forces in discrete differential geometry is the philosophy that 
many concepts and results from classical Riemannian geometry have discrete analogues.
Since intensive research in classical Riemannian geometry has been undertaken for more than a century, the subject of discrete differential geometry is comparably young and attracted interest only for a few decades. Therefore it is comprehensible that many results are well known in a continuous setting but unknown and sometimes even completely out of reach in 
the discrete
 setting.
One might wonder what is the feature that makes 
The challenge on manifolds is to find a suitable function space with appropriate differentiability assumptions. This indeed is trivial for finite graphs since all functions are continuous w.r.t. the discrete topology. However, one of the main challenges in discrete differential geometry is that there is no proper  notion of derivatives and therefore, no chain rule is available for the graph Laplace operator.

We now briefly discuss which results from Riemannian manifolds we aim to transfer to the discrete setting.
One of them is the relation between curvature, volume growth, parabolic Harnack inequality and Gaussian heat kernel estimates.
On Riemannian manifolds, Grigor'yan \cite{Grigoryan91} and Saloff-Coste \cite{saloff1995parabolic} proved the equivalence of parabolic Harnack inequality, Gaussian heat kernel estimates and volume doubling in combination with the Poincar\'e inequality. Due to the celebrated paper by Li and Yau \cite{li1986parabolic}, we know that every complete Riemannian manifold with non-negative Ricci curvature satisfies the parabolic Harnack inequality and therefore, the full characterization of Grigor'yan and Saloff-Coste.
Delmotte \cite{Delmotte1999parabolic} proved that this characterization also holds true in the discrete setting. However it is still unclear whether non-negative Ricci curvature implies Harnack inequalities on graphs.
Another result on Riemannian manifolds we are interested in is the relationship between the Ricci curvature and stochastic completeness.
Intuitively spoken, a Riemannian manifold is stochastically complete if
the probability that the Brownian motion leaves every compact set in finite time is zero.
In particular, it is shown in  \cite{varopoulos1983potential,hsu1989heat} that a decay of Ricci curvature not faster than $-R^2$ implies stochastic completeness.

One of the main difficulties to transfer these results to the discrete setting is to find an appropriate notion of Ricci curvature on graphs.
There are various approaches to do so.
Ollivier defined discrete Ricci curvature via optimal transport \cite{ollivier2009ricci}.
Erbar, Maas and Mielke \cite{erbar2012ricci,mielke2013geodesic} defined discrete Ricci curvature via convexity of the entropy.
The discrete Ricci curvature we consider in this paper is based on the Bakry-\'Emery calculus \cite{BakryEmery85} which has been adopted to the discrete setting in \cite{schmuckenschlager1998curvature,LinYau10}.

Recently, various deep results related to Bakry-\'Emery curvature on graphs have been found.
In \cite{HornLinLiuYau14} it is shown that the exponential curvature dimension condition $\CDE'(0,D)$ implies
the full Delmotte characterization \cite{Delmotte1999parabolic} and in particular Gaussian heat kernel estimates, the parabolic Harnack inequality and the volume doubling. 
In \cite{munch2017remarks} it is shown that $\CDE'(0,D)$ implies 
Bakry and \'Emery's curvature dimension condition $\CCD(0,D).$ Bakry-\'Emery curvature can be reduced to the computation of the smallest eigenvalue of a matrix of reasonable size, and a wide class of examples with non-negative curvature is known, see \cite{cushing2016bakry}.
Naturally, the question arises whether the weaker $\CCD(0,D)$ also implies the full Delmotte characterization. However, this question seems to be hard in general. In this paper, we study a special class of graphs, called linear graphs, which are weighted and undirected graphs supported on the infinite line or its subset.

In the terminology of Markov processes, linear graphs 
associate with birth-death processes (or birth-death chains) according to the weights of the graphs.
We give an affirmative answer to the above question for normalized linear graphs, 
see Corollary~\ref{cor:CDimpliesVD} and Theorem~\ref{thm: CD implies Poincare inequality}.
Interestingly, the dimension parameter $D$ is directly related to the maximal volume growth rate given by $\vol(B_R(x)) \in \mathcal O(R^{D-1})$ where the exponent $D-1$ turns out to be optimal. This relationship is somehow unexpected since the standard lattice $\IZ^D$ satisfies $\CCD(0,2D)$ with optimal dimension parameter $2D$ and has volume growth $\vol(B_R(x)) \in \mathcal O(R^D)$. In this sense, the linear graph has even larger volume growth than the lattice. This contradicts the intuition from manifolds where the Bishop-Gromov theorem implies that $\IR^d$ has the largest volume growth among all $d$-dimensional manifolds with non-negative Ricci curvature.

Another deep result relating to Bakry-\'Emery curvature is that every constant curvature bound in Bakry-\'Emery sense implies stochastic completeness when assuming a non-degenerate vertex measure and a certain completeness assumption \cite{HuaLin17}.
Here it is natural to ask whether the additional assumptions, i.e., non-degenerate vertex measure and completeness are actually needed.
It is well known that in manifolds, even the curvature decaying like $-R^2$ implies stochastic completeness \cite[Theorem~15.4]{Grigoryan99}.
One might wonder if the same holds true in the graph setting.
In this paper, we positively answer both questions for physical linear graphs: Every physical linear graph with Ricci curvature decaying not faster than $-R^2$ is stochastically complete, see Theorem~\ref{thm:scwithdecay}.
While \cite{HornLinLiuYau14} and \cite{HuaLin17} both use heat equation techniques for proving parabolic Harnack inequalities and stochastic completeness, we are able to bypass the use of heat equation and to prove the results in a more direct fashion using a certain local growth rate determined by the curvature.

\subsection{Setup and notation}

Let $(V,E)$ be a connected, undirected, simple graph with the set of vertices $V$ and the set of edges $E.$ Two vertices $x,y$ are called neighbors, denoted by $x\sim y,$ if there is an edge connecting $x$ and $y.$ 
We only consider locally finite graphs, i.e. for each vertex there are only finitely many neighbors.
 We assign a weight $m$ to each vertex, $m: V\to (0,\infty),$ and a weight $w$ to each edge, $$w:E\to (0,\infty),\ E\ni \{x,y\}\mapsto w(x,y)=w(y,x),$$ and refer to the quadruple $G=(V,E,w,m)$ as a \emph{weighted graph}. The weighted graph $(V,E,w,m)$ is called \emph{physical} if $m(x)=1,$ for all $x\in V,$ and is called \emph{normalized} if $m(x)=\sum_{y\in V:y\sim x} w(x,y)$ for all $x\in V.$

A weighted graph $(V,E,w,m)$ is called a \emph{linear graph} (or one-dimensional graph) if $V$ can be identified with $\IZ \cap I$ for a bounded or unbounded interval $I\subset\R$ and $E=\{\{n,n+1\}: n,n+1 \in V\}.$ For simplicity, we write $(V,w,m)$ for a linear graph where the edge set $E$ is clear in the content. Two special cases are of particular interest:
\begin{enumerate}[(a)]
\item For the case $I=\R,$ we say the linear graph is supported on the infinite line, denoted by $(\IZ,w,m).$
\item For the case $I=[0,\infty),$ we say the linear graph is supported on the half line, denoted by $(\IN,w,m).$
\end{enumerate} 
For a physical linear graph $(V,w,m),$ we usually omit the vertex weights $m$ and write $(V,w),$ and sometimes abbreviate $w_n:=w(n,n+1)$ for all $n\in V$ for simplicity.

 For a linear graph $G=(\N_0,w,m)$, we call $d_+(n):= w(n,n+1)/m(n)$ the \emph{out-degree} and $d_-(n):=w(n,n-1)/m(n)$ the \emph{in-degree}.
In terms of Markov processes, the out degree $d_+(n)$ equals the transition rate $Q(n,n+1)$ from $n$ to $n+1$, and similarly, $d_-(n)=Q(n,n-1)$. Moreover, the stationary distribution is given by the vertex measure $m$. Note that we do not require any normalization of the transition rates or of the stationary distribution. In particular, we allow $Q(n,n-1) + Q(n,n+1) \neq 1$ and $m(V)=\sum_x m(x) \neq 1$.

We denote 
by $\ell^p(V,m),$ $p\in [1,\infty],$ the $\ell^p$ spaces of functions on $V$ with respect to the measure $m,$ and by $\|\cdot\|_{\ell^p(V,m)}$ the $\ell^p$ norm of a function.

For a weighted graph $G=(V,E,w,m),$ we define the \emph{Laplacian} $\Delta$ as
$$\Delta f(x)=\frac{1}{m(x)}\sum_{y\in V:y\sim x}w(x,y)(f(y)-f(x)),\quad \forall\ f:V\to\R.$$

Now we introduce the $\Gamma$-calculus on graphs following \cite{BakryGentilLedoux,LinYau10}.
The ``carr\'e du champ" operator $\Gamma$ is defined as
$$\Gamma(f,g)(x):=\frac12(\Delta(fg)-f\Delta g-g\Delta f)(x), \quad \forall f,g:V\to\R, x\in V.$$ For simplicity, we write $\Gamma(f):=\Gamma(f,f).$ The iterated $\Gamma$ operator, $\Gamma_2$, is defined as
$$\Gamma_2(f):=\frac{1}{2}\Delta \Gamma(f)-\Gamma(f,\Delta f).$$



For our purposes, we need the curvature dimension condition on graphs, which was initiated in \cite{BakryEmery85,Bakry87} on manifolds (or Markov diffusion semigroups) and introduced on graphs in \cite{schmuckenschlager1998curvature,LinYau10}. For a weighted graph $G=(V,E,w,m)$ and a vertex $x\in V,$ we say that it satisfies the $\CCD(K,D,x)$ condition, for $K\in \R$ and $D\in (0,\infty],$ if
  $$\Gamma_2(f)(x)\geq \frac1D (\Delta f)^2(x)+K\Gamma (f)(x),\quad \forall\ f:V\to\R.$$ A weighted graph is called satisfying the $\CCD(K,D)$ condition if $\CCD(K,D,x)$ holds for all $x\in V.$ It was proved in \cite{LiuP14,LinLiu15,HuaLin17} that under several conditions,
  $\CCD(0,D)$ holds if and only if for any finitely supported function $f,$
  \begin{equation}\label{eq:gradient1}\Gamma(P_t f)\leq P_t (\Gamma (f)),\end{equation} where $P_t(\cdot)$ denotes the heat semigroup associated to the Laplacian $\Delta.$
For a given graph $G=(V,E,w,m)$ and $\mathcal N >0$ and $x \in V$, we write 
$$\mathcal K_{G,x}(\mathcal N) := \sup \{K \in \IR: \CCD(K,\mathcal N,x)\ \mathrm{holds}\}.$$

We are ready to state our main results.

\subsection{Main results}
We first prove the triviality of edge weights for linear graphs supported on $\IZ$ which satisfy non-negative curvature dimension conditions.

\begin{theorem}\label{thm:Liouville}
Let $(\IZ,w)$ be a physical linear graph satisfying $\CCD(0,\infty).$
Then there is a constant $c>0$ such that $w(n,n+1)=c$ for all $n\in \Z.$ \end{theorem}

The same statement holds true for normalized linear graphs when assuming $\CCD(0,D)$ for a finite dimension $D$.

\begin{theorem}\label{thm:LiouvilleNormaized}
Let $(\IZ,w,m)$ be a normalized linear graph satisfying $\CCD(0,D)$ for some $D<\infty$.
 Then there is a constant $c>0$ such that $w(n,n+1)=c$ for all $n\in \Z.$ \end{theorem}

By these results, the heat semigroup of the  birth-death process on $\Z$ with physical or normalized vertex measure satisfies the gradient bound \eqref{eq:gradient1} if and only if the edge weights are constant. Both above theorems fail for linear graphs supported on $\IN$ instead of $\IZ$. 

By mimicking the definitions in \cite{ollivier2009ricci} and \cite{LinLuYau11}, the second author and Wojciechowski \cite{MuenchWojciechowski17} introduced a generalized Ollivier curvature on the set of edges of a graph, denoted by $\kappa(x,y)$ for each edge $\{x,y\}$, see \cite{MuenchWojciechowski17} and Section~\ref{s:pre} below for details.
We prove that Bakry-\'Emery curvature can be lower bounded by Ollivier curvature if the vertex measure is log-concave.
\begin{theorem}\label{thm:Ollivier}
Let $G=(V,w,m)$ be a linear graph and $n \in V$.
Suppose, the vertex measure is log-concave around $n$, i.e.,
\[
m(n-1)m(n+1) \leq m(n)^2.
\]
Then, $\CCD(K,\infty,n)$ holds with
\[
K= \frac {\kappa(n,n+1) \wedge \kappa(n,n-1)}2.
\]
\end{theorem}

Using this, we construct examples of physical linear graphs with non-trivial edge weights supported on $\IN.$
 
\begin{theorem}\label{thm:concaveweights}
For any positive concave function $f:[0,\infty) \to \R,$ there is $N\in\N$ and a physical linear graph $(\N_0,w)$ satisfying $\CCD(0,\infty)$ such that $w(n,n+1)=f(n+N)$ for all $n\in\IN.$
\end{theorem}
Moreover, we estimate the growth of edge weights for physical linear graphs supported on $\IN$ satisfying curvature dimension conditions,
\begin{theorem}\label{thm:weightgrowth}
	Let $G=(\IN,w)$ be a physical linear graph.
	Then, we have the following:
	\begin{enumerate}
		\item $\CCD(0,\infty)$ implies $w(n,n+1) \in \mathcal O(n),$ as $n\to\infty.$
		\item $\CCD(K,\infty),$ for some $K<0,$ implies $w(n,n+1) \in \mathcal O(n^2)$, as $n\to\infty.$
	\end{enumerate}
\end{theorem}

As a corollary, for physical linear graphs we obtain volume growth properties w.r.t. the intrinsic metric introduced by 
\cite{Huang12,Folz14} under the curvature condition, 
i.e. the polynomial volume growth for linear graphs with non-negative curvature and at most exponential volume growth for those with curvature bounded below by a negative constant, see Corollary~\ref{coro:volumeupper}. Refined volume growth estimates for physical linear graphs with non-negative curvature can be found in Corollary~\ref{coro:refinedest}. Moreover, we prove that the simple random walk on a non-negatively curved physical linear graph is recurrent, see Corollary~\ref{coro:recc}.

By \cite[Theorem~5]{KellerLenz12}, a physical linear graphs is stochastically complete if and only if the following holds,
\begin{align*}
\sum_n \frac n {w(n,n+1)} = \infty.
\end{align*}

Using this criterion, we may prove the stochastic completeness of a physical linear graph whose curvature decays not too fast. This is analogous to the case in the Riemannian setting, see e.g. \cite[Theorem~15.4]{Grigoryan99}. 

\begin{theorem}\label{thm:scwithdecay}

	Let $G=(\IN,w)$ be a physical linear graph and $\rho$ be an intrinsic metric on $G$. We write $K(n)=\mathcal K_{G,n}(\mathcal \infty)$ for any $n\in\IN.$

	Suppose that
	\begin{align*}
	K(n)_- \in \mathcal O(\rho(n,0)^2),
	\end{align*} then $G$ is stochastically complete.
\end{theorem}

Now we turn to normalized linear graphs with non-negative curvature. By introducing the model spaces, we are able to prove a discrete analogue to Bishop-Gromov volume comparison theorem for linear graphs.

\begin{defn}\label{def:model}
	A normalized linear graph $G_0 = (\IN,w_0,m_0)$  is called a \emph{model space} 
	with respect to a \emph{dimension function} $\mathcal N : \IN \to [2,\infty]$ if 
	 for all $x\in \IN$, one has
	$$
	\mathcal K_{G_0,x}(\mathcal N(x)) \leq 0. 
	$$
\end{defn}


\begin{theorem}\label{thm:comparison}
	Let $G_0 = (\IN,w_0,m_0)$ be a model space with respect to a dimension function $\mathcal N : \IN \to [2,\infty]$. Let $G=(V,w,m)$ be a linear normalized graph with $V=\IZ \cap I$  for an interval $I$ and with $\inf V \leq 0$. 
	Suppose for all $x \in \IN \cap V$, one has 
	$$
	\mathcal K_{G_0,x}(\mathcal N(x)) \leq \mathcal K_{G,x}(\mathcal N(x)).
	$$
	Then  we have $$d_+^{G_0}(x) \geq d_+^{G}(x)$$ for all $x \in \IN \cap V$ and
	\begin{align}
	\frac {m_0(y)}{m_0(x)} \geq \frac {m(y)}{m(x)},
	\end{align} for any $x,y\in\IN \cap V$ such that $y>x$.
\end{theorem}

By using the above comparison theorem for a model space with zero curvature, we prove the growth of the measure for normalized linear graphs satisfying $\CCD(0,D)$ for some constant $D$.
\begin{theorem}[Non-negative curvature implies polynomial measure growth]\label{thm: measure growth}
	Let $G=(V,w,m)$ be a normalized linear graph with $V=I \cap \Z$ satisfying $\inf V \leq 0$ and $\sup V = \infty$. Suppose $G$ satisfies $\CCD(0,D)$ for some constant $D\geq 4$. Then, one has
	\begin{align}
	\frac {m(j)}{m(i)} \leq \left(\frac{j+1}{i+1}\right)^{D-2}
	\end{align}
	for all $i,j \in \IN$ with $j>i$.
	Furthermore, $D-2$ is the optimal exponent and $p=d_+ -\frac 1 2$ is non-negative and $m$ is non-decreasing.
\end{theorem}

For a graph $(V,E)$ we denote by $B_R(x)$ the ball of radius $R$ centered at $x\in V$ w.r.t. the combinatorial distance.
\begin{defn}\label{def:volumedoubling}
	Let $G=(V,E,w,m)$ be a graph. We say that $G$ satisfies the volume doubling property with the doubling constant $C,$ denoted by $\mathcal{VD}(C)$, if
	\begin{align}
	m(B_{2R}(x)) \leq C m(B_R(x))
	\end{align}
	for all $x \in V$ and all $R\in \IR_+$.
\end{defn}

Theorem~\ref{thm: measure growth} yields the following corollary.
\begin{corollary}\label{cor:CDimpliesVD}
	Let $G=(V,w,m)$ be an infinite, connected, linear normalized graph satisfying $\CCD(0,D)$ for some $D\geq 4$. Then 
	$\mathcal{VD}\left(2^{D-1}\right)$ holds.
\end{corollary}

\begin{defn}\label{def: poincare}
	We say that a graph $G=(V,E,w,m)$ satisfies the Poincar\'e inequality with constant $C$, denoted by $P(C),$ if
	\begin{align}
	\sum_{x \in B_R(x_0)} m(x)(f(x)-f_B)^2 
	\leq CR^2 \sum_{x,y \in B_{2R}(x_0)} w(x,y)(f(x)-f(y))^2
	\end{align}
	for all $x_0 \in V$ and $R>0$
	with
	$$
	f_B := \frac{1}{m(B_R(x_0))}\sum_{x\in B_R(x_0)}f(x)m(x)
	$$
\end{defn}

Moreover, we prove the Poincar\'e inequality in this setting.
\begin{theorem}\label{thm: CD implies Poincare inequality}
	Let $G=(V,w,m)$ be an infinite, connected, linear normalized graph satisfying $\CCD(0,D)$ for some $D<\infty$. Then $P(16)$ holds.
\end{theorem}

For the Delmotte characterization, we need to assume an ellipticity assumption $E(\alpha)$ on the graph defined by
$$
w(x,y) \geq \alpha m(x)
$$
whenever $x\sim y$.
We will show that normalized linear graphs with $\CCD(0,D)$ satisfy $E(\alpha)$.
\begin{corollary}\label{cor:Ellipticity}
Let $G=(\IN,w,m)$ be a normalized linear graph satisfying $\CCD(0,D)$ for some $D<\infty$. Then, $G$ satisfies $E(\alpha)$ for $\alpha=1/D$.
\end{corollary}

Due to Delmotte \cite[Theorem~1.7]{Delmotte1999parabolic}, the following three properties are equivalent for  all normalized Graph Laplacians.
\begin{enumerate}
	\item The volume doubling property $\mathcal{VD}(C_1)$, the Poincar\'e inequality $P(C_2)$ and ellipticity $E(\alpha)$ for some $C_1,C_2,\alpha>0$.
	\item Parabolic Harnack inequality for the heat semigroup.
	\item Gaussian heat kernel estimate.
\end{enumerate}

So that, we get parabolic Harnack inequality and Gaussian heat kernel estimate by Corollary~\ref{cor:CDimpliesVD}, Theorem~\ref{thm: CD implies Poincare inequality} and Corollary~\ref{cor:Ellipticity}.

Several results can be extended to weakly spherically symmetric graphs, see Corollary~\ref{coro:symm1} and \ref{coro:symm2}. 
In Section~\ref{s:egg}, we construct a class of infinite weighted linear graphs with $\CCD(K,D),$ for $K>0$ and $D<\infty,$ which do not satisfy Feller property, see \cite{wojciechowski2017feller} for definitions.

The paper is organized as follows: In the next section, we introduce combinatorial characterizations of Bakry-\'Emery curvature and Ollivier curvature, and prove Theorem~\ref{thm:Ollivier}. In Section~\ref{s:phy}, we study physical linear graphs and prove Theorem~\ref{thm:Liouville}, \ref{thm:concaveweights}, \ref{thm:weightgrowth} and \ref{thm:scwithdecay}. Section~\ref{s:nor} is devoted to normalized linear graphs where we give the proofs of Theorem~\ref{thm:LiouvilleNormaized}, \ref{thm:comparison}, \ref{thm: measure growth} and \ref{thm: CD implies Poincare inequality} and Corollary~\ref{cor:CDimpliesVD} and \ref{cor:Ellipticity}. In Section~\ref{s:sym} we extend the results to spherically symmetric graphs. Finally in Section~\ref{s:egg}, we give examples which have positive curvature bound, but are non-Feller.

\section{Preliminaries}\label{s:pre}


\subsection{Curvature dimension conditions}
A sufficient and necessary condition for curvature dimension condition can be obtained on linear graphs, see \cite[Theorem~2.5]{HuaLin16}. 
\begin{proposition}\label{prop:equv} Let $G=(V,w,m)$ be a linear graph and $n\in V.$ Then for $K\in \R$ and $D\in (0,\infty],$ $\CCD(K,D,n)$ holds if and only if for all $X,Y\in\R,$
\begin{eqnarray*}&&(1-\frac{2}{D})\left(d_{-}(n)X+d_{+}(n)Y\right)^2\\
&\geq&\frac{d_{-}(n)}{2}\left(d_-(n-1)-3d_+(n-1)+d_-(n)+d_+(n)+2K\right)X^2\\
&+&\frac{d_{+}(n)}{2}\left(d_-(n)+d_+(n)-3d_-(n+1)+d_+(n+1)+2K\right)Y^2\end{eqnarray*}
where we fix $X=0$ if $n-1 \notin V$ making the term $d_\pm(n-1)$ redundant and we fix $Y=0$ if $n+1 \notin V$ making the term $d_\pm(n+1)$ redundant.
  \end{proposition}

From now on, for $n \in V$ and $K\in\R$ and $D\in(0,\infty],$ we set 
$$
W_-(n):= \frac12\left( -d_-(n-1) + 3 d_+(n-1) + \Big( 1- \frac 4 D \Big)d _-(n)-d_+(n)- 2K\right)   
$$
if $n-1 \in V$ and $W_-(n) := 0$ if $n-1 \notin V,$ and set
$$
W_+(n):= \frac12\left( -d_+(n+1) + 3 d_-(n+1) + \Big( 1- \frac 4 D\Big)d _+(n)-d_-(n)- 2K\right)
$$
if $n+1 \in V$ and $W_+(n) := 0$ if $n+1 \notin V$.
 By Proposition~\ref{prop:equv}, for a linear graph $G,$ $\CCD(K,D,n)$ holds if and only if
the following matrix is non-negative, i.e.
\[ \left( \begin{array}{cc}
d_-(n)W_-(n)&  (1-\frac{2}{D})d_-(n)d_+(n) \\
(1-\frac{2}{D})d_-(n)d_+(n)& d_+(n)W_+(n) \end{array} \right)\geq 0.\] 
In the case that $V=\IZ$ and $n\in \IZ,$ or $V=\IN$ and $n\in \N,$ the above is equivalent to 
\begin{equation}\label{eq:equcond}W_-(n)\geq 0\quad \mathrm{and}\quad W_-(n)W_+(n)\geq (1-\frac{2}{D})^2d_{-}(n)d_+(n).\end{equation}

\subsection{Ollivier curvature}
We prove that in case of a log-concave measure, Bakry-\'Emery curvature can be lower bounded by Ollivier curvature.

It is proven in \cite[Theorem~2.8]{MuenchWojciechowski17} that
\[
\kappa(n,n+1)= d_-(n+1) - d_+(n+1) - d_-(n) + d_+(n)
\]
where $\kappa(x,y)$ denotes the Ollivier curvature between two vertices.
Using this formula, we can now prove the connection between Ollivier and Bakry-\'Emery curvature.

\begin{proof}[Proof of Theorem~\ref{thm:Ollivier}]
We aim to prove $\CCD(K,\infty,n)$ with $K=\frac {\kappa(n,n+1) \wedge \kappa(n,n-1)}2$ when assuming $m(n-1)m(n+1) \leq m(n)^2.$
We calculate
\[
2W_-(n) = \kappa(n,n-1) + 2 d_+(n-1) - 2K \geq 2d_+(n-1) = 2\frac{w(n,n-1)}{m(n-1)} >0
\]
and similarly,
\[
2W_+(n) = \kappa(n,n+1) + 2 d_-(n+1) - 2K \geq 2d_-(n+1)
=2\frac{w(n,n+1)}{m(n+1)} >0
\]
Multiplying and applying log-concavity of the measure yields
\[
W_-(n)W_+(n) \geq  \frac{w(n,n-1)}{m(n-1)} \cdot \frac{w(n,n+1)}{m(n+1)} \geq \frac{w(n,n-1)}{m(n)} \cdot \frac{w(n,n+1)}{m(n)} = d_-(n)d_+(n).
\]
This implies $\CCD(K,\infty,n)$ due to Proposition~\ref{prop:equv}.
\end{proof}

\subsection{Intrinsic metrics}\label{s:intrinsic}
The Laplacian associated with the graph $(V,E,w,m)$ is a bounded operator on $\ell^2(m)$ if and only if $$\sup_{x\in V} \frac{\sum_{y\in V:y\sim x}w(x,y)}{m(x)}<\infty,$$ see \cite{KellerLenz12}.
In order to deal with unbounded Laplacians, e.g. Laplacians on physical graphs, we need the following intrinsic metrics introduced in \cite{FrankLenzWingert12}.

A pseudo metric $\rho$ is a symmetric function, $\rho:V\times V\to[0,\infty),$ with zero diagonal which satisfies the triangle inequality. 
We denote by $B_r^\rho(x):=\{y\in V: \rho(y,x)\leq r\}$ the ball w.r.t. the pseduo metric $\rho$ of radius $r$ centered at $x.$

\begin{defn}[Intrinsic metric]\label{d:intrinsic} A pseudo metric $\rho$ on $V$ is called an \emph{intrinsic metric} if
\begin{align*}
\sum_{y\in V:y\sim x}w(x,y)\rho^{2}(x,y)\leq m(x),\qquad \forall x\in V.
\end{align*}
\end{defn}

For any $x,y\in V,$ a walk from $x$ to $y$ is a sequence of vertices $\{x_i\}_{i=0}^k$ such that
$$x=x_{0}\sim\ldots\sim x_{k}=y.$$ Here $k$ is called the length of the walk. We denote by $P_{x,y}$ the set of walks from $x$ to $y.$
Let $d$ denote the combinatorial distance on a (connected) graph $(V,E),$ i.e. for any $x,y\in V,$
$d(x,y)$ is the minimal length of the walks from $x$ to $y.$ As is well-known, the combinatorial distance $d$ is an intrinsic metric for normalized graphs.

Intrinsic metrics always exist on a general graph $(V,E,w,m).$ There is a natural intrinsic metric introduced by  \cite{Huang12,Folz14}.
Define the \emph{weighted vertex degree} $\mathrm{Deg}:V\to[0,\infty)$ by
\begin{align*}
    \mathrm{Deg}(x)=\frac{1}{m(x)}\sum_{y\in V:y\sim x}w(x,y),\qquad x\in V.
\end{align*}

\begin{example}\label{ex:intrinsic}
For any weighted graph, there is an intrinsic path metric defined by
\begin{equation}\label{def:huangetal}
    \sigma(x,y)=\inf_{\{x_i\}_{i=0}^n\in P_{x,y}}\sum_{i=0}^{n-1} (\mathrm{Deg}(x_{i})\vee\mathrm{Deg}(x_{i+1}))^{-\frac{1}{2}}.
\end{equation}
\end{example} 

\section{Physical linear graphs}\label{s:phy}

In this section, we study curvature dimension conditions on physical linear graphs.  Let $(V,w)$ be a physical linear graph. For convenience, we write $w_n:=w(n,n+1),$ for all $n\in V$ and set $w_{-2}=w_{-1}=0$ and $m_{-2}=m_{-1}=1$ if it is supported on $\IN.$

Given a physical linear graph $G=(V,w),$ $K\in\R$ and $D\in(0,\infty],$ for any $n\in V,$
$$2W_-(n)=  -w_{n-2} + \Big( 4- \frac 4 D \Big)w_{n-1}-w_n-2K
$$ and
$$
2W_+(n)=-w_{n+1} + \Big( 4- \frac 4 D\Big)w_n-w_{n-1}- 2K.
$$

By \eqref{eq:equcond}, we have the following proposition.
\begin{proposition}\label{prop:phys} Let $G=(V,w)$ be a physical linear graph. Suppose $n\in \IZ$ for $V=\Z,$ or $n\in\N$ for $V=\IN,$ then $\CCD(K,\infty,n)$ holds if and only if $$4w_{n-1}-(w_{n-2}+w_n)-2K\geq 0$$ and
\begin{equation}\label{equ4}[4w_{n-1}-(w_{n-2}+w_n)-2K][4w_{n}-(w_{n-1}+w_{n+1})-2K]\geq 4w_{n-1}w_{n}.\end{equation}\end{proposition}

The following proposition gives a sufficient condition for $\CCD(0,\infty)$ using Theorem~\ref{thm:Ollivier}.
\begin{proposition}\label{prop:new} Let $(\IN,w)$ be a physical linear graph and $n\in\N.$ If the Ollivier curvature is non-negative, i.e.,
$$\kappa(i,i+1) = 2 w_{i} -w_{i-1} - w_{i+1} \geq 0,\quad  i=n,n-1,$$ then $\CCD(0,\infty,n)$ holds. 
\end{proposition}
\begin{proof} 
This follows immediately from Theorem~\ref{thm:Ollivier} since physical linear graphs have a log-concave vertex measure.
\end{proof}


We recall some basic facts on concave functions. Let $f:[0,\infty)\to(0,\infty)$ be a positive concave function. 
Then $f$ is monotonely non-decreasing, the limit $$A:=\lim_{t\to\infty}f_{-}'(t)\ \mathrm{exists},\quad A\in[0,\infty),$$ and for any $A_1>A$ there exist $C_1,C_2, t_0>0$ such that
\begin{equation}\label{equ1}A t+C_1\leq f(t)\leq A_1t+C_2,\quad \forall\ t\in[t_0,\infty).\end{equation}

For any positive concave function on $[0,\infty),$ one can construct a physical linear graph $(\N_0,w)$ with weights given by the shifted function.
\begin{proof}[Proof of Theorem~\ref{thm:concaveweights}] We claim that there exists $N\geq2,$ such that $f(N+1)\leq 4 f(N).$ Otherwise, there exists $k\in\N$ such that  $f(n+1)>4 f(n)$ for $n\geq k.$ This yields that $f(n)\geq 4^{n-k} f(k),$ and contradicts to \eqref{equ1}. This proves the claim.

We define \[g(t)=\left\{\begin{array}{cc}(t+1) f(N),&t\in[-1,0],\\f(t+N),&t>1.\end{array}\right.\] Since $f$ is concave on $[N-1,\infty),$ $g$ is concave on $[-1,\infty).$ Then we define a linear graph on $\N_0\cup\{-1\},$ such that $w_{n}=g(n)$ for $n\geq -1.$
Note that $w_{-1}=0,$ so that $\{-1\}$ is an isolated vertex. One can check that it satisfies $\CCD(0,\infty)$ by Proposition~\ref{prop:new} for $n\in\N$ and Proposition~\ref{prop:equv} for $n=0.$
\end{proof}


We define a function $\phi$ on $(0,\infty)$ by
\begin{equation}\label{equ2}\phi(x)=\left\{\begin{array}{cc}\frac{1}{4}\left(7-\frac{9}{4x-1}\right),& x>\frac14,\\
-\infty,& 0<x<\frac14.\end{array}\right.\end{equation}
\begin{lemma}\label{lem1} Let $(\IN,w)$ be a physical linear graph, $n\in\N_0$ and $n\geq 2.$ Suppose that $\CCD(0,\infty,n)$ holds, 
then
$$\frac{w_{n+1}}{w_n}\leq \phi(\frac{w_{n-1}}{w_{n-2}}),$$ where $\phi$ is defined in \eqref{equ2}.
\end{lemma}
\begin{proof} For convenience, we set $(a,b,c,d)=(w_{n-2},w_{n-1},w_n,w_{n+1}).$ By \eqref{eq:equcond}, $4b-a-c\geq0,$  which yields that $a<4b.$
And \eqref{equ4} reads as $$(4b-a-c)(4c-b-d)\geq 4bc>0.$$ Hence
$$d\leq 4c-b-\frac{4bc}{4b-a-c}.$$ Dividing $c$ in both sides, 
$$\frac{d}{c}\leq 4-\frac{b}{c}-\frac{4b}{4b-a-c}=:H(a,b,c).$$ For fixed $a,b,$ $H(a,b,c)\leq H(a,b,c_0)$ where $c_0=\frac{4b-a}{3}$ given by $\partial_cH(a,b,c_0)=0.$ Hence $$\frac{d}{c}\leq H(a,b,c_0)=\frac{1}{4}\left(7-\frac{9}{\frac{4b}{a}-1}\right)=\phi(\frac{b}{a}). $$ This proves the lemma.
 
 \end{proof}

We collect some basic properties of the function $\phi:$ \begin{enumerate}\item It is concave and monotonely increasing on $(1/4,\infty).$ \item $\phi(1)=1,$ $\phi'(1)=1.$ \item $\frac{\phi(x)}{x}$ is non-decreasing on $(0,1]$ and bounded above by $1.$ \end{enumerate}

\begin{theorem}\label{thm:onN} Let $(\IN,w)$ be a physical linear graph satisfying $\CCD(0,\infty).$ Then $w_i$ is non-decreasing on $i\in \N_0.$
\end{theorem}
\begin{proof} Without loss of generality, we prove $w_0\leq w_1$ since the same argument applies for any $w_n\leq w_{n+1}$ with $n\in\N.$

We argue by contradiction. Suppose that $w_0>w_1.$ Then by Lemma \ref{lem1},
$$\frac{w_{2i+1}}{w_{2i}}\leq \phi(\frac{w_{2i-1}}{w_{2i-2}}),\quad i\geq 2.$$ Hence for any $n\in \IN,$ $w_{2n+1}<w_{2n}$ and
$$\frac{w_{2n+1}}{w_{2n}}\leq \phi^{(n)}(\frac{w_1}{w_0}),$$ where $\phi^{(n)}$ is the $n$th-composition of function $\phi.$ 
By the property of the function $\phi,$ for $x_0\in(0,1),$ $$\phi(x)\leq \frac{\phi(x_0)}{x_0}x,\quad \forall\ 0<x<x_0.$$ This yields that $$\frac{w_{2n+1}}{w_{2n}}\leq \delta^n \frac{w_1}{w_0},$$ where $\delta=\frac{w_0}{w_1}\phi(\frac{w_1}{w_0})<1.$ For sufficiently large $n,$ we have $\delta^n\frac{w_1}{w_0}<\frac14.$ So that $\frac{w_{2n+3}}{w_{2n+2}}\leq-\infty.$ This is a contradiction. This proves the theorem.
\end{proof}

By this result, we can prove that physical linear graphs on $\IZ$ are trivial.
\begin{proof}[Proof of Theorem~\ref{thm:Liouville} (for physical linear graphs)] Note that the same argument as in the proof of Theorem~\ref{thm:onN} implies that $w_n$ is non-decreasing for $n\in \IZ.$ By reflecting the graph w.r.t. $0,$ we obtain that $w_n$ is non-increasing, so that $w$ is constant on $\IZ.$
\end{proof}

We have a basic lemma on intrinsic metrics on physical linear graphs.

\begin{lemma}\label{lem:intrinsiclinear}
	Let $\rho$ be an intrinsic metric on a physical linear graph $(\N_0,w)$. Then,
	\begin{align}
	\rho(0,n) \leq \sum_{k=0}^{n-1} {\frac 1 {\sqrt{w_k}}} 
	\end{align}
\end{lemma}
\begin{proof}
	For $k\in \N_0,$
	since $\rho$ is intrinsic, we have
	$$
	w_k\rho(k,k+1)^2 \leq 1
	$$
which implies $\rho(k,k+1) \leq \frac{1}{\sqrt{w_k}}$. The claim follows from the triangle inequality.
\end{proof}

As a corollary of Theorem~\ref{thm:onN}, we prove that any physical linear graph supported on $\N_0$ with non-negative curvature has at least linear volume growth w.r.t. any intrinsic metric.
\begin{corollary}\label{coro:lowerlinear} Let $G=(\IN,w)$ be a physical linear graph satisfying $\CCD(0,\infty)$ and $\rho$ be an intrinsic metric on $G.$  Then $m(B_r^\rho(0))\geq \lfloor\sqrt{w_0}\cdot r\rfloor ,$ for any $r\in \N.$
\end{corollary}
\begin{proof} Since $w_i$ is non-decreasing on $i\in\N_0,$ $w_i\geq w_0,\ \forall i\in \N.$ By Lemma~\ref{lem:intrinsiclinear}, for any $n\in \N,$ $$\rho(0,n)\leq \sum_{k=0}^{n-1} {\frac 1 {\sqrt{w_k}}}\leq \frac{n}{\sqrt{w_0}}.$$

This implies that $[0,\sqrt{w_0}r]\cap \N_0\subset B_r^{\rho}(0)$ and proves the corollary.
\end{proof}

\subsection{Completeness and Stochastic completeness}
The completeness of a weighted graph was introduced by \cite{HuaLin17} to mimic the completeness of a Riemannian manifold. A weighted graph $(V,E,\mu,m)$ is called \emph{complete} if there is a non-decreasing sequence of finitely supported functions $\{\eta_k\}_{k=1}^\infty$ such that \begin{equation}\lim_{k\to\infty}\eta_k(x)=1\ \mathrm{and}\ \ \Gamma(\eta_k)(x)\leq \frac{1}{k},\quad \forall x\in V.\end{equation}

We will give a sufficient and necessary condition for the completeness of a physical linear graph.
\begin{theorem}
	Let $G=(\IN,w)$ be a physical linear graph. The following are equivalent:
	\begin{enumerate}
		\item $G$ is complete.
		\item \begin{align*}
		\sum_n \frac 1 {\sqrt{w_n}} = \infty.
		\end{align*}
	\end{enumerate}
\end{theorem}

\begin{proof}
	$(1)\Rightarrow (2)$:
	
	Let $k \in \N$. Due to the completeness, there exist finitely supported functions $\eta_k$ on $\IN$ s.t. $\Gamma( \eta_k) \leq \frac 1 k.$ Without loss of generality, we may assume that $\eta_k(0)=1$. 
	Then, $$w_n(\eta_k({n+1}) - \eta_k(n))^2 \leq \Gamma (\eta_k) (n) \leq 1/k$$ and therefore,
	\begin{align*}
	1 \leq \sum_n |\eta_k(n+1) - \eta_k(n)| \leq \sum_n \frac 1 {\sqrt {k w_n}}
	\end{align*}
	which implies
	\begin{align*}
	\sqrt{k}\leq  \sum_n \frac 1 {\sqrt {w_n}}.
	\end{align*}		
	Since $k$ can be arbitrary large,we obtain assertion (2) of the theorem.

	$(2)\Rightarrow (1)$:
	We split the sum of $1/\sqrt{w_n}$ into two parts, one part with even $n$ and the other part with odd $n$. At least one of these sums is infinite due to (2).
	We first suppose the sum with odd $n$ is infinite. We first suppose
	\begin{align*}
	\sum_n \frac 1 {\sqrt{w_{2n+1}}} = \infty.
	\end{align*}
	Now we define $\eta_k$ inductively via
	\begin{itemize}
		\item $\eta_k(0):=1,$
		\item $\eta_k(n+1) := \eta_k(n)$ if $n$ even,
		\item $\eta_k(n+1) := \left[ \eta_k(n) - \frac 1 {\sqrt{k w_n}} \right]_+$ if $n$ odd.
	\end{itemize}
	It is easy to see that the construction implies $\Gamma(\eta_k) \leq \frac 1 k$ and that $\eta_k$ converges to $1$ pointwise and that $\eta_k$ is increasing in $k$, even without assuming assertion (2) of the theorem.
	It is left to show that $\eta_k$ is finitely supported. 
	The induction yields
	\begin{align*}
	\eta_k(n)_=   \left( 1 -  \sum_{0 \leq i \leq n/2 -1}  \frac 1 {\sqrt{kw_{2i+1}}} \right)_+
	\end{align*}
	Therefore, for large $n$, we have 
	$\eta_k(n)=0$ since
	\begin{align*}
	\sum_{0 \leq i \leq n/2 -1}  \frac 1 {\sqrt{kw_{2i+1}}} \geq 1
	\end{align*}
	for $n$ large enough due to assertion (2) of the theorem.
	This shows that $\eta_k$ is finitely supported.
	The case that the sum with even $n$ is infinite works similarly.
	This finishes the proof.	
\end{proof}

\subsection{Non-negative curvature on physical graphs}

\begin{theorem}\label{thm: discrete second derivative and curvature}
	Let $G=(\IN,w)$ be a physical linear graph. Let $n \in \N,$ $n\geq 2$, and $K\leq 0$. Suppose $\CCD(K,\infty, n)$ and $\CCD(K,\infty, n+1)$ hold. Then,
	\begin{align*}
	w_{n-2} - 2w_n + w_{n+2} \leq -6K.
	\end{align*}
\end{theorem}

\begin{proof}
We write $(a,b,c,d,e)=(w_{n-2},w_{n-1},w_{n},w_{n+1},w_{n+2}).$

Then $\CCD(K,\infty,n)$  implies
\begin{align*}
(4b-a-c-K)(4c-b-d-K) \geq 4 bc
\end{align*}
and $\CCD(K,\infty,n+1)$ implies
\begin{align*}
(4c-b-d-K)(4d-c-e-K) \geq 4 cd
\end{align*}
where all the factors are positive.

Hence,
\begin{align*}
a \leq 4b-c-K - \frac {4bc}{4c-b-d-K}
\end{align*}
and 
\begin{align*}
e \leq 4d-c-K  - \frac{ 4 cd}{4c-b-d-K} 
\end{align*}

Summing up yields
\begin{align*}
a+e-2c &\leq 4(b+d)- 4c- 2K  - \frac{ 4 c(b+d)}{4c-(b+d)-K} \\
&=-6K -4 \cdot \frac{(4c-b-d-K)(c-b-d-K)+c(b+d)}{4c-(b+d)-K}\\
&=-6K -4 \cdot \frac{(4c-b-d-K)(c-b-d-K)+c(b+d+K) - cK}{4c-(b+d)-K}\\
&=-6K -4 \cdot \frac{(2c-b-d-K)^2- cK}{4c-(b+d)-K}\\
&\leq -6K.
\end{align*}
where the last identity follows from $(4c-x)(c-x)+cx=(2c-x)^2$ with $x=b+d+K$ and the inequality follows from a positive denominator and from $K\leq 0$ and $c\geq 0$.
\end{proof}

This theorem gives strong implications on the growth of edge weights.
\begin{proof}[Proof of Theorem~\ref{thm:weightgrowth}]

Applying Theorem~\ref{thm: discrete second derivative and curvature} inductively yields, for any $N\geq 1$,
$$
w_{2N+2} - w_{2N} \leq w_2 - w_0  - 6KN
$$
and therefore.
$$
w_{2N+2} \leq w_2 + N(w_2 - w_0 - 6KN) \in 
\begin{cases}
\mathcal{O}(N) &: K=0 \\
\mathcal{O}(N^2) &: K<0 .\\
\end{cases}
$$
The odd case $w_{2N+1}$ works analogously which finishes the proof.
\end{proof}

Note that by Theorem~\ref{thm:onN}, the intrinsic metric $\sigma,$ defined in \eqref{def:huangetal}, is given by $$\sigma(n,n+1)=\frac{1}{\sqrt{w_n+w_{n+1}}},\quad n\in\IN.$$ 
Theorem~\ref{thm:weightgrowth} yields the volume growth of balls w.r.t. the intrinsic metric  $\sigma.$
\begin{corollary}\label{coro:volumeupper} Let $(\IN,w)$ be a physical linear graph and $\sigma$ be the intrinsic metric defined in \eqref{def:huangetal}. Then we have the following:
\begin{enumerate}\item $\CCD(0,\infty)$ implies that $m(B_r^\sigma(0))\in \mathcal{O}(r^2),$ as $r\to \infty.$
\item $\CCD(K,\infty),$ for some $K<0$, implies that $m(B_r^\sigma(0))\in \mathcal{O}(e^{Cr}),$ as $r\to \infty.$
\end{enumerate}

\end{corollary}
\begin{proof}For the first case, by Theorem~\ref{thm:onN} and Theorem~\ref{thm: discrete second derivative and curvature}, for any $k\in\N$ we have
$$\sigma(k,k+1)=\frac{1}{\sqrt{w_k+w_{k+1}}}\geq \frac{1}{\sqrt{2w_{k+1}}}\geq \frac{C}{\sqrt{k}}.$$ Hence $$\sigma(0,n)\geq C+\sum_{k=1}^{n-1}\frac{C}{\sqrt{k}}\geq C_1\sqrt{n}.$$ Hence $$B_r^{\sigma}(0)\subset [0,C_1^{-2}r^2].$$ This prove the first case.

For the second case, we have $w_k \in \mathcal O(k^2)$ by Theorem~\ref{thm: discrete second derivative and curvature} and therefore,
$$\sigma(k,k+1)=\frac{1}{\sqrt{w_k+w_{k+1}}}\ \geq \frac{C}{k}.$$
Hence, 
$$
\sigma(0,n)\geq C+\sum_{k=1}^{n-1}\frac{C}{k}\geq C_1\log{n}
$$
which yields
$$
B_r^\sigma(0) \subset \left[0,e^{r/C_1} \right].
$$
This proves the second case and finishes the proof of the theorem.
\end{proof}

In order to describe the edge weights quantitatively, we introduce the following definitions. 
\begin{defn}\label{def:fgfunction}For a physical linear graph $G=(\IN,w),$ we define two positive functions associated to $G$ as
\[f_G:[0,\infty)\to(0,\infty),\quad f_G(x)=\left\{\begin{array}{ll}w_{2n},& x=2n,n\in\IN,\\
\mathrm{linear\ interpolation},& \mathrm{otherwise}.\end{array}\right.\]
\[g_G:[1,\infty)\to(0,\infty),\quad g_G(x)=\left\{\begin{array}{ll}w_{2n+1},& x=2n+1,n\in\IN,\\
\mathrm{linear\ interpolation},& \mathrm{otherwise}.\end{array}\right.\]
\end{defn} For simplicity, we write $f=f_G$ and $g=g_G.$
For a physical linear graph $G=(\IN,w)$ satisfying $\CCD(0,\infty),$ by Theorem~\ref{thm:onN} and Theorem~\ref{thm: discrete second derivative and curvature}, $f$ and $g$ are non-decreasing concave functions.

Note that by the monotonicity of $w_n,$
\begin{equation}\label{eq:monotone}f(2n)\leq g(2n+1)\leq f(2n+2)\leq g(2n+3),\quad \forall n\in \IN.\end{equation} It is obvious that $\{w_n\}_{n=1}^\infty$ is bounded if and only if $f(x)$ and $g(x)$ is bounded, in this case $$\lim_{x\to\infty}f(x)=\lim_{x\to\infty}g(x)=\lim_{n\to\infty}w_n,$$ 
 and $\{w_n\}_{n=1}^\infty$ is unbounded if and only if $\lim_{x\to \infty}f(x)=\lim_{x\to \infty}g(x)=\infty.$ By the linear interpolation for $f$ and $g,$ for any $n\in\IN,$ \begin{eqnarray}f'_-(x)&=&\frac12(w_{2n+2}-w_{2n}),\quad \forall x\in(2n,2n+2], \label{eq:fderiv}\\ g'_-(x)&=&\frac12(w_{2n+3}-w_{2n+1}),\quad \forall x\in(2n+1,2n+3].\label{eq:gderiv}\end{eqnarray} By \eqref{eq:monotone}, one can show that $f$ and $g$ have same asymptotic behavior at infinity.
\begin{proposition}Let $G=(\IN,w)$ be a physical linear graph satisfying $\CCD(0,\infty)$ and $f$ and $g$ be two concave functions associated to $G$ defined as in Definition~\ref{def:fgfunction}. Then
\begin{equation}\label{eq:agconst}\lim_{x\to\infty}f_-'(x)=\lim_{x\to\infty}g_-'(x)=\lim_{x\to\infty}\frac12(w_{2n+2}-w_{2n})=\lim_{x\to\infty}\frac12(w_{2n+3}-w_{2n+1})=:A_G.\end{equation}
\end{proposition}
\begin{proof} Since $f$ and $g$ are positive concave, the following limits exist $$A_1:=\lim_{x\to\infty}f_-'(x)\ \ \mathrm{and}\ \ A_2:=\lim_{x\to\infty}g_-'(x).$$ We want to show that $A_1=A_2.$ Suppose this is not true. Without loss of generality, we may assume that $A_1<A_2.$ Then by \eqref{equ1}, there exist $x_0, C_1, C_2>0$ such that
$$f(x)\leq \frac12(A_1+A_2)x+C_2, \ g(x)\geq A_2 x+C_1,\quad \forall x\geq x_0.$$
By \eqref{eq:monotone}, for sufficiently large $n\in \N,$
$$A_2(2n-1)+C_1\leq g(2n-1)\leq f(2n)\leq \frac12(A_1+A_2)(2n)+C_2.$$ This yields a contradiction as $n\to\infty.$ This proves that $A_1=A_2.$ For the other part of the theorem, it follows from \eqref{eq:fderiv}, \eqref{eq:gderiv} and the fact that $f_-'$ and $g_-'$ are non-increasing.
\end{proof}

\begin{proposition}\label{prop:oneconcave}Let $G=(\IN,w)$ be a physical linear graph satisfying $\CCD(0,\infty)$ and $f$ and $g$ be two concave functions associated to $G.$ Then
$$\limsup_{x\to\infty}|f(x)-g(x)|\leq 3A_G,$$
where $A_G$ is defined as in \eqref{eq:agconst}. In particular, there exist $n_0\in \N, C>0$ such that 
\begin{equation}\label{eq:closeconcave}|w_n-f(n)|\leq C,\quad \forall n\geq n_0.\end{equation}
\end{proposition}
\begin{rem} As before, we used two concave functions $f$ and $g$ to describe the edge weights $w_n$ in the Definition~\ref{def:fgfunction}. The above proposition shows that it suffices to use a single concave function, up to some constant error, see \eqref{eq:closeconcave}.
\end{rem}

\begin{proof} 
It suffices to consider the first part of the proposition since the second part follows directly from the first.
 For any $n\in \N$ and $x\in [2n,2n+1],$ we have \begin{eqnarray*}f(x)-g(x)&\leq&f(2n+2)-g(2n)=w_{2n+2}-g(2n)\leq g(2n+3)-g(2n)\\&=&3g'(\zeta_1),\quad \exists\zeta_1\in[2n,2n+3],\end{eqnarray*} and 
\begin{eqnarray*}f(x)-g(x)&\geq&f(2n)-g(2n+1)=w_{2n}-g(2n+1)\geq g(2n-1)-g(2n+1)\\&=&-2g'(\zeta_2),\quad \exists\zeta_2\in[2n-1,2n+1].\end{eqnarray*} Hence $$|f(x)-g(x)|\leq 3\max_{\zeta\in[2n-1,2n+3]}g_-'(\zeta).$$ Similarly, for $x\in [2n+1,2n+2],$ we have
$$|f(x)-g(x)|\leq 3 \max_{\zeta\in[2n-1,2n+3]}g_-'(\zeta).$$ 
The first part of the proposition follows from \eqref{eq:agconst} by taking $x\to \infty.$ 
\end{proof}

For two positive functions $h,k$ on $[0,\infty),$ we write $$h(x) \in \Theta(k(x)),\quad \mathrm{as}\ x\to\infty,$$ if there exist positive constants $C_1,C_2,x_0$ such that $$C_1 k(x)\leq h(x)\leq C_2 k(x), \quad\ \mathrm{for}\ x\geq x_0.$$ Let $G=(\IN,w)$ be a physical linear graph and $\sigma$ be the intrinsic metric defined in \eqref{def:huangetal}. We say that 
\begin{itemize} \item $G$ has linear volume growth (w.r.t. the metric $\sigma$) if $m(B_r^\sigma(0))\in\Theta(r),\ \mathrm{as}\ r\to\infty.$
\item $G$ has intermediate volume growth if $$\lim_{r\to\infty}\frac{m(B_r^\sigma(0))}{r}=\infty,\quad \lim_{r\to\infty}\frac{m(B_r^\sigma(0))}{r^2}=0.$$
\item $G$ has quadratic volume growth if $m(B_r^\sigma(0))\in\Theta(r^2),\ \mathrm{as}\ r\to\infty.$
\end{itemize} The estimates in Corollary~\ref{coro:volumeupper} can be refined as follows.

\begin{corollary}\label{coro:refinedest} Let $G=(\IN,w)$ be a physical linear graph satisfying $\CCD(0,\infty)$ and $\sigma$ be the intrinsic metric defined in \eqref{def:huangetal}. Then we have the following:
\begin{enumerate}[(a)]
\item $G$ has linear volume growth if and only if $\{w_n\}_{n=0}^\infty$ is bounded.
\item $G$ has intermediate volume growth if and only if $\{w_n\}_{n=0}^\infty$ is unbounded and $$\lim_{n\to\infty}w_{2n+2}-w_{2n}=0.$$
\item $G$ has quadratic volume growth if and only if $$\lim_{n\to\infty}w_{2n+2}-w_{2n}>0.$$
\end{enumerate}
\end{corollary}

\begin{proof} Let $f$ and $g$ be two concave functions associated to $G$ defined as in Definition~\ref{def:fgfunction}.

$(a)$ If $\{w_n\}_{n=0}^\infty$ is bounded, then it is easy to obtain the upper bound estimate for linear volume growth. The lower bound estimate for linear volume growth follows from Corollary~\ref{coro:lowerlinear}.

$(b)$ If $\{w_n\}_{n=0}^\infty$ is unbounded and $\lim_{n\to\infty}w_{2n+2}-w_{2n}=0,$ then $\lim_{x\to\infty}f(x)=\infty$ and $\lim_{x\to\infty}f_-'(x)=0.$ By Proposition~\ref{prop:oneconcave}, there is $n_0\in\N$ and $C>0$ such that 
\begin{equation}\label{eq:wnest}f(n)-C\leq w_n\leq f(n)+C,\quad \forall n\geq n_0.\end{equation}

We first show that $\lim_{r\to\infty}\frac{m(B_r^\sigma(0))}{r^2}=0.$ There exists $n_1>n_0$ such that for any $k
\geq n_1,$ $$\sigma(k,k+1)=\frac{1}{\sqrt{w_k+w_{k+1}}}\geq\frac{1}{\sqrt{2w_{k+1}}}\geq\frac{1}{\sqrt{2f(k+1)+2C}}\geq\frac{C_1}{\sqrt{f(k+1)}}.$$ Hence for any $n>n_1,$ \begin{equation}\label{eq:est1}\sigma(0,n)\geq C_2+\sum_{k=n_1}^{n-1}\frac{C_1}{\sqrt{f(k+1)}}\geq C_2+\int_{n_1}^n\frac{C_1}{\sqrt{f(x+1)}}dx,\end{equation} where we use the fact that $f$ is non-decreasing. Set $F(t)=C_2+\int_{n_1}^t\frac{C_1}{\sqrt{f(x+1)}}dx.$ Denote by $F^{-1}(\cdot)$ the inverse function of $F$ and write $r=F(t)$ and $t=F^{-1}(r).$ We claim that $\lim_{t\to\infty} F(t)=\lim_{r\to\infty}F^{-1}(r)=\infty.$ In fact,  by L'H\^opital's rule, 
$$\lim_{t\to\infty}\frac{F(t)}{\sqrt{t}}=\lim_{t\to\infty}\frac{\frac{C_1}{\sqrt{f(t+1)}}}{\frac{1}{2\sqrt t}}=\lim_{t\to\infty}C\sqrt{\frac{t}{f(t+1)}}.$$ Note that $\lim_{t\to\infty}f'_-(t)=0,$ by L'H\^opital's rule, $$\lim_{t\to\infty}\frac{t}{f(t+1)}=\lim_{t\to\infty}\frac{1}{f'_-(t+1)}=\infty.$$ This yields \begin{equation}\label{eq:est2}\lim_{t\to\infty}\frac{F(t)}{\sqrt{t}}=\infty\end{equation} and we prove the claim. By the estimate \eqref{eq:est1}, we get
$$B_r^\sigma(0)\subset [0, F^{-1}(r)].$$ In particular, $m(B_r^\sigma(0))\leq F^{-1}(r).$ This yields that, by \eqref{eq:est2},
$$\lim_{t\to\infty}\frac{m(B_r^\sigma(0))}{r^2}\leq \lim_{r\to\infty}\frac{F^{-1}(r)}{r^2}=\lim_{t\to\infty}\frac{t}{F^2(t)}=0.$$ This proves the result.

Next, we show that  $\lim_{r\to\infty}\frac{m(B_r^\sigma(0))}{r}=\infty.$ 
By \eqref{eq:wnest}, there exists $n_1>n_0$ such that for any $k
\geq n_1,$ $$\sigma(k,k+1)\leq\frac{1}{\sqrt{2w_{k+1}}}\leq\frac{1}{\sqrt{2f(k)-2C}}\leq\frac{C_3}{\sqrt{f(k)}}.$$ Hence for any $n>n_1,$ \begin{equation}\label{eq:est4}\sigma(0,n)\leq C_4+\sum_{k=n_1}^{n-1}\frac{C_3}{\sqrt{f(k)}}\leq C_4+\int_{n_1-1}^{n-1}\frac{C_3}{\sqrt{f(x)}}dx.\end{equation} Set $F(t)=C_4+\int_{n_1}^t\frac{C_3}{\sqrt{f(x-1)}}dx.$ This yields that $$B_r^{\sigma}(0)\supset [0,F^{-1}(r)].$$

As before, one can show that $\lim_{t\to\infty} F(t)=\lim_{r\to\infty}F^{-1}(r)=\infty.$ Moreover,  by L'H\^opital's rule, $$\lim_{r\to\infty}\frac{F^{-1}(r)}{r}=\lim_{t\to\infty}\frac{t}{F(t)}=\lim_{t\to\infty}\frac{\sqrt{f(t-1)}}{C_3}=\infty.$$ This implies 
$$\lim_{r\to\infty}\frac{m(B_r^\sigma(0))}{r}\geq \lim_{r\to\infty}\frac{\lfloor F^{-1}(r)\rfloor}{r}\geq\lim_{r\to\infty}\frac{F^{-1}(r)-1}{r}=\infty.$$ This proves that $G$ has intermediate volume growth.


$(c)$ If $\{w_n\}_{n=0}^\infty$ is unbounded and $\lim_{n\to\infty}w_{2n+2}-w_{2n}>0.$  Then $\lim_{x\to\infty}f(x)=\infty$ and $A_G=\lim_{x\to\infty}f_-'(x)>0.$ The upper bound estimate for quadratic volume growth has been obtained in Corollary~\ref{coro:volumeupper}. To get the lower bound estimate for quadratic volume growth, it suffices to estimate the volume growth as in $(b)$ by using \eqref{equ1}, i.e. there exists $x_0$ and $C>0$ such that $$f(x)\geq A_Gx-C, \quad x\geq x_0.$$

The corollary follows from combining $(a)-(c).$
\end{proof}

\begin{corollary}\label{coro:recc} Any physical linear graph satisfying $\CCD(0,\infty)$ is recurrent.
\end{corollary}	
\begin{proof} By Theorem~\ref{thm:Liouville}, it suffices to consider $(\IN,w).$ We denote by $R(0,n)=\sum_{k=0}^{n-1}\frac{1}{w_k}$ the resistance between $0$ and $n\in \N.$ Then by Theorem~\ref{thm:weightgrowth},  
$$R(0,n)\geq \frac{1}{w_0}+\sum_{k=1}^{n-1}\frac{C}{k}\to\infty,\quad \mathrm{as}\ n\to\infty.$$This yields the recurrence of the random walk by Nash-Williams criterion \cite{nash1959random}.
\end{proof}


In the following, we study the stochastic completeness of linear graphs with the curvature decaying to $-\infty$.
We will prove that if the curvature decays not faster than $-\rho^2$ where $\rho$ is an intrinsic metric, then the graph is stochastically complete.



\begin{proof}[Proof of Theorem~\ref{thm:scwithdecay}]
	We will prove 
	\begin{align*}
	\sum_n \frac n {w_n} = \infty.
	\end{align*}
	which characterizes stochastic completeness on physical linear graphs.

	The proof is based on a case distinction in terms of $\rho(n,0)$.
	
	First case:
	\begin{align*}
	\rho(n,0)^2 \notin \mathcal O(\log n).
	\end{align*}
	In this case, we do not need any curvature assumption. We only use Cauchy-Schwarz-inequality:
	\begin{align*}
	\rho(0,n)^2 \leq \left( \sum_{k=0}^{n-1} {\frac 1 {\sqrt{w_k}}} \right)^2 &\leq C +  \left( \sum_{k=1}^{n-1} {\frac {k} {{w_k}}} \right) \cdot  \left( \sum_{k=1}^{n-1} \frac 1 {k}  \right)\\
	&\leq C +  \left( \sum_{k=1}^{n-1} {\frac {k} {{w_k}}} \right) \cdot  \log(n).\\
	\end{align*}
	Since $\rho(n,0)^2 \notin \mathcal O(\log n)$, this implies 
	\begin{align*}
	\infty = \limsup_{n\to \infty} \frac{\rho(0,n)^2}{\log(n)} \leq \limsup_{n\to \infty} \sum_{k=1}^{n-1} {\frac {k} {{w_k}}}. 
	\end{align*}
	Hence, $G$ is stochastically complete in the first case, i.e., if $\rho(n,0)^2 \notin \mathcal O(\log n).$
	
	We now consider the second case
	\begin{align*}
		\rho(n,0)^2 \in \mathcal O(\log n).
	\end{align*}
Note that we assume that $K(n)_- \in \mathcal O(\rho(n,0)^2).$
	This implies
	$$(w_{n-2} - 2 w_n +w_{n+2})_+ \in \mathcal{O}( \log(n))$$ due to Theorem~\ref{thm: discrete second derivative and curvature} and the decay assumption on the curvature.
	Summing up yields $(w_{n+2}-w_n)_+ \in \mathcal{O}(n \log(n))$ and summing up again yields
	$w_n \in \mathcal{O}(n^2 \log(n))$.
	Thus, there exist $C>0$ and $N\in \N$ s.t.
	\begin{align*}
	\sum_n \frac n{w_n} \geq C \sum_{n\geq N} \frac 1{n \log n} = \infty.
	\end{align*}
	This proves stochastic completeness in the second and final case.
	Thus, the proof is finished.
\end{proof}

\section{Normalized linear graphs}\label{s:nor}

Recently, it was proven that the exponential curvature dimension inequality $\CDE'(0,D)$ implies the volume doubling, the Poincar\'e inequality, the parabolic Harnack inequality and Gaussian heat kernel estimates (see \cite{HornLinLiuYau14}). In \cite{munch2017remarks} it is shown that $\CDE'(0,D)$ implies $\CCD(0,D)$ which is a linear condition computable via semi-definite programming (see \cite{cushing2016bakry}).
It turns out to be a hard but interesting question whether already $\CCD(0,D)$ implies the volume doubling, the Poincar\'e inequality, etc..
In this section we  answer this question affirmatively for the class of normalized linear graphs.

Our approach is to prove the volume growth and the Poincar\'e inequality via a certain local growth rate which we can control via Proposition~\ref{prop:equv}.
Due to Delmotte's characterization (see \cite{Delmotte1999parabolic}), this implies the parabolic Harnack inequality and Gaussian heat kernel estimates.

For a normalized linear graph $G=(\IN,w,m)$, the \emph{local growth rate} $p:\IN \to [-\frac 1 2 , \frac 1 2]$ is defined as $p := d_+ - \frac 1 2$, i.e., $p(n) = \frac{w(n,n+1)}{m(n)} - \frac 1 2$.
It is easy to see that $p=0$ (except finitely many vertices) means linear grwoth, $p<0$ means sublinear growth, and $p\geq C>0$ means exponential growth. We will later show that $p(n) \sim \frac 1 n$ means polynomial growth which we are interested in this section.

Our first aim is to characterize curvature conditions via local growth rates. To do so, we introduce functions $\mathcal F$ and $\mathcal G$ with arguments $a,b$ standing for local growth rates, $K$ the curvature bound and $D$ the dimension bound.

\begin{defn}\label{def:F,G}
	We write
	\begin{align*}
	\Fc(a,b,K,D):= \left(1- \frac 2 D\right)\left( \frac 1 2-b\right) + 2a - K
	\end{align*}
	and
	\begin{align*}
	\Gc(a,b,K,D):=\left(\frac 1 2 + b \right)\left[1-\frac 2 D  - \frac{\left(1 - \frac 2 D \right)^2 \left(\frac 1 2 -b\right)}{\Fc(a,b,K,D)}\right]  - K.
	\end{align*}
Thereby, we set the fractions to be zero whenever the numerator is zero and we set them to minus infinity if only the denominator is zero.	
\end{defn}

Remark that $\Gc(a,b,K,D)$ is well defined when $b=1/2,$ even if $a$ is not defined.


\begin{theorem}\label{thm:CDcharNormalized}
	Let $G=(V,w,m)$ be a linear, normalized graph with $V=\IZ \cap I$ for an interval $I$.
	Let $n \in V$. Suppose that also $n+1 \in V$. The following are equivalent
	\begin{enumerate}
		\item $G$ satisfies $\CCD(K,D,n)$.
		\item One has $\Fc(p(n-1),p(n),K,D) \geq 0$ and 
		\begin{align*}
		2p(n+1) \leq \Gc(p(n-1),p(n),K,D).
		\end{align*}
		where we set $\Fc(p(n-1),p(n),K,D) \geq 0$ to be true if $n-1 \notin V$.
	\end{enumerate}
	
\end{theorem}
Remark that $\Gc(p(n-1),p(n),K,D)$ is well defined even if $n-1 \notin V$ since this can only happen if $p(n)=1/2$ and in this case, the value of $\mathcal G$ does not depend on the value of $p(n-1)$.

\begin{proof}

According to the remark after Proposition~\ref{prop:equv}, we set 
$$
2W_-(n)=  -d_-(n-1) + 3 d_+(n-1) + \Big( 1- \frac 4 D\Big)d _-(n)-d_+(n) - 2K
$$
and
$$
2W_+(n)= -d_+(n+1) + 3 d_-(n+1) + \Big( 1- \frac 4 D\Big)d _+(n)-d_-(n)  - 2K.
$$
We set $W_-(n) = 0$ if $n-1 \notin V$.
Due to Proposition~\ref{prop:equv} and the remark subsequently, $\CCD(K,D)$ holds  if and only if
$W_-(n)\geq 0$ and  $W_+(n)\geq 0$ and
$$
4W_-(n)W_+(n)\geq 4\left(1 - \frac 2 D \right)^2 d_+(n)d_-(n).
$$
Now suppose that the graph is normalized.
We have $d_+(k)= \frac 1 2 + p(k)$. Thus if $\inf V < n < \sup V$,
$$
2W_-(n)= \Big( 2- \frac 4 D \Big) \left(\frac 1 2 - p(n)  \right) + 4 p(n-1) - 2K
$$
and
$$
2W_+(n)= \Big( 2- \frac 4 D\Big) \left(\frac 1 2 + p(n)  \right) - 4 p(n+1) - 2K.
$$

We observe $W_-(n) = 2\Fc(p(n-1),p(n),K,D)$ if $n-1,n \in V$.

We already know that for $n \in V$ the curvature condition $\CCD(K,D,n)$ is equivalent to $W_-(n) \geq 0$ and $W_+(n) \geq 0$ and
$$W_-(n)W_+(n) \geq 4\left(1 - \frac 2 D \right)^2 \left(\frac 1 2 +p(n)\right)\left(\frac 1 2 -p(n)\right).$$
But this is, under condition of $W_-(n) \geq 0$, equivalent to
\begin{align*}
4p(n+1) \leq \left(2- \frac 4 D\right)\left( \frac 1 2 + p(n)\right) - 2K - \frac{4\left(1 - \frac 2 D \right)^2 \left(\frac 1 2 +p(n)\right)\left(\frac 1 2 -p(n)\right)}{W_-(n)}
\end{align*}

where we set the latter summand to be zero whenever the numerator $$4\left(1 - \frac 2 D \right)^2 \left(\frac 1 2 +p(n)\right)\left(\frac 1 2 -p(n)\right)$$ is zero.

We set $b:=p(n)$ and $a:=p(n-1)$ and equivalently reformulate
\begin{align*}
2p(n+1) &\leq \left(\frac 1 2 + b \right)\left[1-\frac 2 D  - \frac{\left(1 - \frac 2 D \right)^2 \left(\frac 1 2 -b\right)}{\Fc(a,b,K,D)}\right]  - K\\
&= \Gc(a,b,K,D).
\end{align*}
This finishes the proof.
\end{proof}

\subsection{Bishop-Gromov volume comparison}

We will prove that if a graph has larger curvature than another graph, then it has slower volume growth than the other.

The assumptions $K\leq 0$ and $D>2$ in Definition~\ref{def:model} are necessary to characterize the model space property via the functions $\Fc$ and $\Gc$ (see Definition~\ref{def:F,G}).

We now give monotonicity  properties of the functions $\Fc$ and $\Gc$ defined in Definition~\ref{def:F,G}.
This will allow us to characterize the model space property.

\begin{lemma}\label{lem:MonotoneG}
Suppose $\Fc(a,b,K,D) \geq 0$ and $|b| \leq 1/2$. Then
\begin{enumerate}
\item
$\Gc(a,b,K,D)$ is strictly decreasing in $K$ and increasing in $a$.
\item If $D\geq 2$ and $K\leq 0$ and $a\geq 0$, then
$\Gc(a,b,K,D)$ is increasing in $b$ and $D$. Furthermore, $\Gc(a,b,K,D) \geq 0$.
\item If $b=1/2$, then $\mathcal G(a,b,K,D)$ does not depend on $a$.
\end{enumerate}
\end{lemma}
\begin{proof}
	This can be easily checked by taking derivatives of $\Gc$.
\end{proof}

By the characterization of the curvature dimension condition, we have control on the optimal curvature bound $\mathcal K_{G,n}(D)$. 
\begin{theorem}\label{thm:CharModelSpace}
	Let $G=(V,w,m)$ be a normalized linear graph and let $D>0$. Let $K \in \IR$ and $n \in V$. Suppose $n+1 \in V$.
	Suppose $\mathcal F(p(n-1),p(n),K,D) \geq 0 $ or $n-1 \notin V$.
	The following are equivalent
	\begin{enumerate}
	\item
		$$ 2p(n+1) = \Gc(p(n-1),p(n), K,D).$$
		\item 		$$K = \mathcal K_{G,n}(D).$$
	\end{enumerate}

\end{theorem}

\begin{proof}
	This follows from Theorem~\ref{thm:CDcharNormalized} and strict monotonicity  of $\mathcal G$ in $K$.
\end{proof}

\begin{rem}
	For $n=0$, the value $p(n-1)$ is undefined. However since $p(0)=\frac 1 2$, the value $\Gc(p(n-1),p(n),\mathcal K_{G,n}(D),D)$ does not depend on $p(n-1)$ as stated in Lemma~\ref{lem:MonotoneG}.
\end{rem}

We remind for linear graphs $G$, we write $d_+(x)=w(x+1,x)/m(x)$. We write $d_+^{G}:=d_+$ to indicate the underlying graph $G$.


\begin{proof}[Proof of Theorem~\ref{thm:comparison}]
	We write $p_0(n):=\frac{w_0(n,n+1)}{m_0(n)} - \frac 1 2$ and
	$p(n):=\frac{w(n,n+1)}{m(n)} - \frac 1 2.$
	
	We prove via induction that $p_0(n) \geq p(n)$ and $p_0(n) \geq 0$ for all $n \in \IN \cap V$.
	
	We have $p_0(0)=\frac 1 2 \geq p(0)$. We can assume without obstruction that $p_0(-1)=\frac 1 2 \geq p(-1)$ since the curvature of $G_0$ at $n=0$ does not depend on $p(-1)$ and $p_0(-1)$.

	Moreover if we assume $0 \leq p_0(n-1) \geq p(n-1)$ and $p_0(n) \geq p(n)$ for some $n \geq 0$, then Lemma~\ref{lem:MonotoneG} and Theorem~\ref{thm:CharModelSpace}  yield
	\begin{align*}
	2p(n+1) &= \Gc(p(n-1),p(n),\mathcal K_{G,n}(\mathcal N (n)),\mathcal N(n)) \\
	&\leq \Gc(p_0(n-1),p(n),\mathcal K_{G_0,n}(\mathcal N(n)),\mathcal N(n)) \\
	&\leq \Gc(p_0(n-1),p_0(n),\mathcal K_{G_0,n}(\mathcal N(n)),\mathcal N(n)) \\	
	&=2p_0(n+1)
	\end{align*}
where the first inequality follows from the first part of Lemma~\ref{lem:MonotoneG} and the second inequality from the second part of Lemma~\ref{lem:MonotoneG} using $p_0(n-1) \geq 0$ and $K_{G_0,n}(\mathcal N(n)) \leq 0$ and $\mathcal N(n) \geq 2$.
This proves $p_0(n+1) \geq p(n+1)$. Furthermore, 
$$
2p_0(n+1)= \Gc(p_0(n-1),p_0(n),\mathcal K_{G_0,n}(\mathcal N(n)),\mathcal N(n)) \geq 0
$$
where the inequality follows from the second part of Lemma~\ref{lem:MonotoneG} and by induction assumption $p_0(n-1)\geq 0$.
The finishes the induction and proves $d_+^{G_0} \leq d_+^G$.	
	Hence, 
\begin{align*}
\frac {m_0(y)}{m_0(x)} 
&=\frac{\frac 1 2 + p_0(x)}{\frac 1 2 - p_0(x+1)} \cdot \ldots \cdot \frac{\frac 1 2 + p_0(y-1)}{\frac 1 2 - p_0(y)}
\\&\leq \frac{\frac 1 2 + p(x)}{\frac 1 2 - p(x+1)} \cdot \ldots \cdot \frac{\frac 1 2 + p(y-1)}{\frac 1 2 - p(y)} \\&= \frac {m(y)}{m(x)}
\end{align*}
which finishes the proof.	
\end{proof}



We give the model space for $\mathcal K_{G,x}(D)=0$  for all $x \in V$ and given $D > 2$. By Bishop-Gromov theorem, this will give us sharp volume growth bounds.

\begin{defn}
	Let $D>2$.  We define the linear normalized graph $G_D:=(\IN,w_D,m_D)$ where
	$w_D$ and $m_D$ are uniquely determined by
	\begin{align*}
	p_D(n):=d_+^{G_D}(n) - \frac 1 2 = \frac{D-2}{2D+4(n-1)}.
	\end{align*}
	
\end{defn}

\begin{rem}
	Indeed, $G_D$ is well defined since $d_+(0)=1$.
\end{rem}

We will show that $G_D$ is a model space with $K=0$ and dimension $D$ which will allow us to give sharp volume growth bounds.

\begin{theorem}\label{thm: model space D const}
Let $D>2$. Then, $\mathcal K_{G_D,n}(D)=0$ for all $n \in \IN$.
\end{theorem}
\begin{proof}
By Theorem~\ref{thm:CharModelSpace}, it suffices to show for all $n\geq 0$
\begin{align*}
0< 2p_D(n+1) &= \Gc(p_D(n-1),p_D(n),0,D) \\& 
= \left(\frac 1 2 + b \right) \cdot \frac{2a\left( 1-\frac 2 D\right)}{\left(1 - \frac 2 D  \right)\left(\frac 1 2 - b \right) + 2a}
\end{align*}
with $a:=p_D(n-1)$ and $b:=p_D(n)$ (see the defintion of the function $\Gc$ in Definition~\ref{def:F,G}).
We have
$\frac 1 2 + b = \frac{D+n-2}{D+2(n-1)}$ and $\frac 1 2 - b = \frac n {D+ 2 (n-1)}$ and $2a = \frac{D-2}{D+2(n-2)}$.
Hence,
\begin{align*}
&\left(\frac 1 2 + b \right) \cdot \frac{2a\left( 1-\frac 2 D\right)}{\left(1 - \frac 2 D  \right)\left(\frac 1 2 - b \right) + 2a} \\
=&\frac{D+n-2}{D+2(n-1)} \cdot \frac{\frac{D-2}{D+2(n-2)} \left(1 - \frac 2 D \right)}{\left(1 - \frac 2 D \right) \frac n {D+ 2 (n-1)}   + \frac{D-2}{D+2(n-2)}} \\
=&(D+n-2) \frac{D-2}{n(D+2(n-2)) + D(D+2(n-1))}\\
=&(D+n-2) \frac{D-2}{(n+D)(D+2n) -4n +2D}\\
=&\frac{D-2}{D+2n}\\
=&2p_D(n+1)
\end{align*}
as desired.
\end{proof}

We remind, for a linear, normalized graph $G=(\IN,w,m)$, the value $p(n)=d_+(n) - \frac 1 2 = \frac{w(n+1,n) + w(n-1,n)}{m(n)} - \frac 1 2$ determines the volume growth.
Positive $p$ means increasing measure $m$ and negative $p$ means decreasing $m$. If $p(n)>C>0$, then the volume grows exponentially. The following theorem gives upper bounds for $p$ under a curvature dimension condition.


\begin{theorem}\label{thm:pnExplicitStrong} 
Let $G=(V,w,m)$ be a linear normalized graph satisfying $\CCD(0,D)$ for some $D>2$. 
Suppose $\inf V \leq 0$.
Then
for all $n \in \IN \cap V$, one has	
\begin{align*}
p(n) \leq p_{D}(n)=\frac{D-2}{2D+4(n-1)}.
\end{align*}
\end{theorem}
\begin{proof}
This directly follows from Theorem~\ref{thm: model space D const}	and Theorem~\ref{thm:comparison}.
\end{proof}

We now translate the growth rate $p$ into the growth of the measure and prove Theorem~\ref{thm: measure growth}.

\begin{proof}[Proof of Theorem~\ref{thm: measure growth}]
	Due to induction principle, it suffices for the first part to prove \begin{align*}
	\frac {m(n+1)}{m(n)} \leq \left(\frac{n+2}{n+1} \right)^{D-2}
	\end{align*}
	for all $n \in \IN$.	
	By assumption, $D\geq 4$ and thus, $D-3 \geq 1$ and $D+2n-2 \geq 2(n+1)$. Hence, 
\begin{align*}
	\frac {m(n+1)}{m(n)} &= \frac {d_+(n)}{d_-(n+1)} = \frac{\frac 1 2 + p(n)}{\frac 1 2 - p(n+1)} \leq \frac{\frac 1 2 + p_{D}(n)}{\frac 1 2 - p_{D}(n+1)} \\ 
	&= \frac{1+ \frac{D-2}{D+2n-2}}{1 - \frac{D-2}{D+2n}}\\
	&=1 + \frac{D-2}{n+1} + \frac{D-2}{(D+2n-2)(n+1)}\\
	&\leq 1 + \frac{D-2}{n+1} + \frac{(D-2)(D-3)}{2(n+1)^2}\\
	&\leq \left(1+\frac{1}{n+1} \right)^{D-2}\\
	&=\left(\frac{n+2}{n+1} \right)^{D-2}.
	\end{align*}
	We now show optimality of the exponent $D-2$.
	Let $D' < D$ and suppose $G=G_D$. By the above calculation, we have
	\begin{align*}
	\frac{m(n+1)}{m(n)} = 1 + \frac{D-2}{n+1} + \frac{D-2}{(D+2n-2)(n+1)} \geq 1 + \frac{D-2}{n+1} > \left(\frac{n+2 } {n+1} \right)^{D'-2}
	\end{align*}
	for large $n$.
	This shows optimality of the exponent $D-2$.

	It is left to show that $m$ is non-decreasing.
	To do so, we consider the reflected graph $\widetilde G = (\widetilde V, \widetilde w, \widetilde m)$ given by renaming the vertices via an isomorphism $\Phi: V\to \widetilde V, n \mapsto C-n$ for some fixed $C \in \IN$. We denote $\widetilde d_+$ and $\widetilde d_-$ and $\widetilde p$ accordingly.
	Observe $\widetilde p(C-n) = -p(n)$ since $\widetilde d_+(C-n) = d_-(n)$.
	Moreover, $\inf(\widetilde V) = -\infty \leq 0$ due to the assumption that $\sup V = \infty$.
Applying Theorem~\ref{thm:pnExplicitStrong} to $\tilde G$ which also satisfies $\CCD(0,D)$ yields
$$-p(n) = \widetilde p(C-n) \leq p_D(C-n) \stackrel{C \to \infty}{\longrightarrow} 0.$$
This proves $p(n) \geq 0$ which immediately implies that $m$ is non-decreasing. This finishes the proof.	
\end{proof}

\begin{proof}[Proof of Theorem~\ref{thm:LiouvilleNormaized}]
Due to Theorem~\ref{thm: measure growth}, we have that $m(n)$ is non-decreasing in $n$. Due to symmetry of the vertex set $\IZ$, this also tells us
that $m(n)$ is non-increasing in $n$ and therefore needs to be constant. This immediately implies that the edge weights are the same for the whole graph.
\end{proof}

\begin{proof}[Proof of Corollary~\ref{cor:Ellipticity}]
We aim to show $w(x,y)/m(x) \geq 1/D$ whenever $x\sim y$.
since $m(x)$ is non-decreasing, we can assume without obstruction that $y=x-1 \geq 0$.
Due to Theorem~\ref{thm:pnExplicitStrong}, we have
\begin{align*}
\frac {w(x,x-1)}{m(x)} = d_-(x) = \frac 1 2 - p(x) &\geq \frac 1 2 - \frac{D-2}{2D+4(x-1)}\\
& \geq \frac 1 2 - \frac{D-2}{2D+4(1-1)}\\
&=\frac 1 D.
\end{align*}
This finishes the proof.
\end{proof}

In order to prove the volume doubling property, we define the sphere doubling property which turns out to be sharply preserved when taking Cartesian products.
It will also imply the volume doubling and follow from $\CCD(0,D)$.

\begin{defn}
	We say, a graph $G=(V,w,m)$ satisfies the sphere doubling property, with constant $C$ (called $\mathcal{SD}(C)$), if for all $i,j \in \IN$ with $j \leq 2i+1$ and all $x_0 \in V$, one has 
	\begin{align*}
		m(S_{j}(x_0)) \leq C	m(S_i(x_0)). 
	\end{align*}
\end{defn}

\begin{theorem}[Sphere doubling implies volume doubling]\label{thm: SD implies VD}
	Let $G=(V,w,m)$ be a graph with $\mathcal{SD}(C)$. Then, $G$ satisfies $\mathcal{VD}(2C)$.
\end{theorem}
\begin{proof}
	We have
	\begin{align*}
		m(B_{2R+1}(x)) &= \sum_{k=0}^R m(S_{2k}(x)) + m(S_{2k+1}(x)) \\ 
		&\leq  \sum_{k=0}^R 2Cm(S_k(x)) \\&= 2Cm(B_R(x))
	\end{align*}
	which finishes the proof.
\end{proof}

\begin{theorem}[Non-negative curvature implies sphere doubling]\label{thm: CD implies SD} 
	Let $G=(\IN,w,m)$ be an infinite linear connected normalized graph satisfying $\CCD(0,D)$ for some $D\geq4$. Then, $G$ satisfies $\mathcal{SD}(2^{D-2})$.
\end{theorem}
\begin{proof}
	We have to show $m(S_j(x)) \leq 2^{D-2}m(S_i(x))$ for all $x\in V$ whenever $0 \leq j \leq 2i+1$.
	For convenience, we set $m(x):=0$ for $x<0$. Thus by Theorem~\ref{thm: measure growth}, we see that $m$ is non-decreasing on $\IZ$.
	
	We prove by case analysis. W.l.o.g., $i\neq j$.
	First, suppose $i=0$. Then we may assume $j = 1$.
	If $x=0$, we have $m(S_j(x))=m(1) \leq 2^{D-2} m(0) = 2^{D-2} m(S_i(x))$.
	If $i=0$ and $x\geq 1$, we have 
	\begin{align*}
		m(S_j(x))&=m(x+1) + m(x-1) \leq m(x+1) +m(x) \leq \left[1+\left(\frac{x+2}{x+1}\right)^{D-2} \right] m(x) \\
		&\leq  \left[1+\left(\frac{3}{2}\right)^{D-2} \right] m(x)	
		\leq 2^{D-2} m(x) = 2^{D-2}m(S_i(x)).
	\end{align*}
	
	For $2i+1 \geq j\geq i\geq 1$, we have by Theorem~\ref{thm: measure growth},
	\begin{align*}
		m(S_j(x)) = m(x+j)+m(x-j) &\leq \left(\frac{x+j+1}{x+i+1}\right)^{D-2} m(x+i) + m(x-i) \\&\leq 
		2^{D-2} m(x+i)+ m(x-i) \\&\leq 2^{D-2}m(S_i(x)).
	\end{align*}
	
	For the degenerated case $j \leq i$, we have
	\begin{align*}
		m(S_j(x)) \leq  m(x+j) + m(x-j) \leq 2m(i) \leq 2^{D-2}m(S_i(x)).
	\end{align*}
	
	Thus we have considered all cases which finishes the proof.
\end{proof}

By Theorem~\ref{thm: CD implies SD}, the $\CCD$ condition implies $\mathcal{SD}$ and by Theorem~\ref{thm: SD implies VD}, $\mathcal{SD}$ implies $\mathcal{VD}$. Hence, the $\CCD$ condition implies the volume doubling as stated in the following corollary.

\begin{proof}[Proof of Corollary~\ref{cor:CDimpliesVD}]
By Theorem~\ref{thm: CD implies SD} , $\CCD(0,D)$ implies $\mathcal{SD}(2^{D-2})$.
Furthermore, by Theorem~\ref{thm: SD implies VD}, $\mathcal{SD}(2^{D-2})$ implies $\mathcal{VD}(2^{D-1})$. 
 \end{proof}

It turns out that $\mathcal{SD}$ is better compatible with taking Cartesian products than $\mathcal{VD}$ since the product of balls is a $l_\infty$ ball and the ball in the product is a $l_1$ ball and there seems to be no possibility to sharply compare these different balls.

\begin{defn}
	Let $G_i=(V_i,w_i,m_i)$ be graphs ($i=1,2$).
	We write $G_1 \times G_2 := (V_1 \times V_2, w_{12},m_{12})$ with
	\begin{align*}
	w_{12}((x_1,x_2),(y_1,y_2)) := w_1(x_1,y_1) 1_{x_2=y_2} + 
	w_1(x_2,y_2) 1_{x_1=y_1}
	\end{align*}
	and
	$m_{12}(x_1,x_2) := m_1(x_1) m_2(x_2)$.
\end{defn}

We now show that taking Cartesian products preserves sphere doubling.
\begin{theorem}
	Let $G_1,G_2$ be graphs with $\mathcal{SD}(C_1)$ and $\mathcal{SD}(C_2)$ respectively. Then, the Cartesian product $G_1 \times G_2$ satisfies $\mathcal{SD}(2C_1C_2)$.
\end{theorem}

\begin{proof}
	Let $x_i \in V_i$ for $i=1,2$.
	We have
	\begin{align*}
		S_n{(x_1,x_2)} = \bigsqcup_{k=0}^n S_k(x_1) \times S_{n-k}(x_2).
	\end{align*}
	Let $i \in \IN$.
	Hence for $j\leq 2i+1$,
	\begin{align*}
		m(S_j{(x_1,x_2)}) &= \sum_{k=0}^j m(S_k(x_1)) m(S_{j-k}(x_2))\\
		&\leq \sum_{k=0}^j C_1 m\left(S_{\left\lceil \frac {k-1} {2} \right\rceil}(x_1)\right) C_2 m\left(S_{\left\lceil \frac {2i-k} {2} \right\rceil}(x_2)\right)\\
		&\leq \sum_{k=0}^{2i+1} C_1 m\left(S_{\left\lceil \frac {k-1} {2} \right\rceil}(x_1)\right) C_2 m\left(S_{\left\lceil \frac {2i-k} {2} \right\rceil}(x_2)\right)\\
		&=2C_1C_2 \sum_{k=0}^i m(S_k(x_1)) m (S_{i-k}(x_2))\\
		&=2C_1C_2 m (S_i(x_1,x_2))
	\end{align*}
\end{proof}

\begin{rem}
	The first inequality in the proof is sharp whenever the sphere doubling of $G_1$ and $G_2$ is sharp.
	The second inequality is sharp if $j=2i+1$.
\end{rem}

\subsection{Poincar\'e inequality}

We use the Cheeger-inequality $\lambda_1 \geq \frac 1 2 h^2$ to prove the Poincaré inequality (see Definition~\ref{def: poincare}). To do so, we will calculate
the Cheeger constant of subgraphs of the underlying linear graph.
More precisely, we calculate Cheeger constants of balls:

\begin{defn}
	Let $G=(V,w,m)$ be a graph and $W \subset V$. We define the restriction
	$G_W:=(W,w_W,m_W)$ with $w_W(x,y):=w(x,y)1_W(x)1_W(y)$ and $m_W:=m1_W$.
\end{defn}

In case of linear graphs, this means the following.
Let $A<B \in \IN$
We write $G_{A,B}:=G_{[A,B] \cap \IN} =(V_{A,B},w_{A,B},m_{A,B})$ with $V_{A,B}=[A,B]\cap \IN$ and $w_{A,B}(i,j)=w(i,j)1_{[A,B]}(i)1_{[A,B]}(j)$ and $m_{A,B}(i)=m(i)1_{[A,B]}(i)$.

For finite graphs $G=(V,w,m)$, we define the Cheeger constant
$h:=h(G):=\inf_{A \subset V} \frac{w(A,A^c)}{m(A)\wedge m(A^c)}$.

\begin{lemma}\label{lemma: Cheeger intervals}
	Let $G=(\IN,w,m)$ be a linear normalized graph with $p = d_+ -\frac 1 2\geq 0$. Let $A<B \in \IN$. Then,
	\begin{align*}
	h(G_{A,B}) \geq \frac 1 {2(B-A)}.
	\end{align*}
\end{lemma}
\begin{proof}
Since we can assume that $A$ and $A^c$ are connected and non-empty, this means for $G_{A,B}$ that

\begin{align*}
h(G_{A,B}) &= \inf_{n\in V_{A,B-1}} \frac{w(n,n+1)}{m([A,n]) \wedge m([n+1,B])}\\
&\geq \inf_{n\in V_{A,B-1}} \frac{w(n,n+1)}{m([A,n])}
\end{align*}

where we write $m([a,b]) :=m([a,b]\cap \IN)$.

Since $p\geq 0$, we have that $m$ is increasing which implies for $n \leq B-1$ that
$$m[A,n] \leq (n-A+1)m(n) = (n-A+1)\frac{w(n,n+1)}{d_+(n)} \leq 2(B-A)w(n,n+1).$$
Thus, 
\begin{align*}
h(G_{A,B}) \geq \frac 1 {2(B-A)}
\end{align*}
which finishes the proof.
\end{proof}
We use Cheeger's inequality $\lambda_1(G_{A,B}) \geq \frac 1 2 h(G_{A,B})^2$ to estimate the spectral gap which will give us the Poincaré inequality.


The next lemma shows that the Poincar\'e inequality holds true whenever the Cheeger constant of balls can be lower bounded by the radius.

\begin{lemma}\label{lemma: Cheeger Poincare}
	Let $G=(V,w,m)$ be a graph. Suppose there exists $c>0$ s.t. $h(G_{B_R(x_0)}) \geq c/R$. Then, $G$ satisfies the Poincar\'e inequality $P\left(1/c^2\right)$.
\end{lemma}
\begin{proof}
	Due to Cheeger's inequality, we have 
	$$\lambda_1(G_{B_R(x_0)}) \geq  \frac 1 2 h(G_{B_R(x_0)})^2 \geq \frac {c^2}{2R^2}$$
	Let $x_0 \in V$ and $R>0$ and $f:V\to \IR$. We write
	$g:= (f-f_B)|_{B_R(x_0)}$.
	Then, $g$ is orthogonal to $1$ on $G_{B_R(x_0)}$.
	The Min-max principle tells us that $\mathcal{E}(g) \geq \lambda_1(G_{B_R(x_0)}) \left\| g \right\|_2^2$ with
	$\mathcal E (g):= \frac 1 2 \sum_{x,y\in B_{R}(x_0)}w(x,y)(g(x)-g(y))^2$ and $\left\| g \right\|_2^2 = \sum_{x\in B_R(x_0)} m(x)g(x)^2$.
	Hence,
	\begin{align*}
	\sum_{x \in B_R(x_0)} m(x)(f(x)-f_B)^2 &= \left\| g \right\|_2^2 \leq \frac{1}{\lambda_1(G_{B_R(x_0)})} \mathcal E(g) \\
	&= 	\frac{1}{\lambda_1(G_{B_R(x_0)})} \cdot \frac 1 2 \sum_{x,y \in B_{R}(x_0)} w(x,y)(f(x)-f(y))^2\\
	&\leq \frac{R^2}{c^2} \sum_{x,y \in B_{2R}(x_0)} w(x,y)(f(x)-f(y))^2.
	\end{align*}
Thus, $G$ satisfies $P(1/c^2)$.
\end{proof}

Now we are ready to prove Theorem~\ref{thm: CD implies Poincare inequality}.
\begin{proof}[Proof of Theorem~\ref{thm: CD implies Poincare inequality}]
	Let $x_0 \in V=\IN$ and let $R>0$. Then $B_R(x_0) = [x_0-R,x_0+R] \cap \IN = \{A,\ldots,B\}$ for some $A<B\in \IN$ with $B-A \leq 2R$.
	Due to Theorem~\ref{thm: measure growth}, $p=d_+-\frac 1 2\geq 0$.
	Thus by Lemma~\ref{lemma: Cheeger intervals}, $$h(G_{A,B}) \geq \frac 1 {2(B-A)} \geq \frac 1 {4R}.$$
	By Lemma~\ref{lemma: Cheeger Poincare}, this implies $P(16)$.
\end{proof}
We remark that the Poincar\'e inequality on linear graphs encodes a certain uniformity of the volume growth which seems to be rather weak (it already follows from $p\geq 0$ on linear normalized graphs) and therefore, it does not depend on the dimension bound.




\section{Applications}
As further applications of the theory of one-dimensional graphs, we generalize the volume doubling and the stochastic completeness to weakly spherically symmetric graphs. Moreover, we construct infinite linear graphs with a uniform positive curvature bound.

\subsection{From linear to weakly spherically symmetric graphs}\label{s:sym}

Weakly spherically symmetric graphs have been introduced in \cite{KellerLenzWojciechowski10} and can be seen as a generalization of linear graphs.
We show that curvature bounds transfer from linear to weakly symmetric graphs.

\begin{defn}
	Let $G=(V,w,m)$ be a graph and let $x_0 \in V$.
	We write 
	$$d_+(y):=\sum_{z:d(z,x_0)>d(y,x_0)} w(y,z)$$
	and
	$$d_-(y):=\sum_{z:d(z,x_0)<d(y,x_0)} w(y,z).$$	
	We say a graph $G=(V,w,m)$ is \emph{weakly spherically symmetric} w.r.t. $x_0 \in V$
	if $d_+(y_1)=d_+(y_2)$ and $m(y_1)=m(y_2)$ whenever $d(x_0,y_1)=d(x_0,y_2)$.
	We say a graph $G=(V,w,m)$ is \emph{normalized symmetric} w.r.t. $x_0 \in V$
	if $G$ is weakly spherically symmetric and if  $d_+(y)+d_-(y)=1$ for all $y$.
	We say a graph $G=(V,w,m)$ is \emph{physically symmetric} w.r.t. $x_0 \in V$
	if $G$ is weakly spherically symmetric and if  $m(S_n(x_0))=1$ for all $n \in \IN$.

	If $G$ is weakly spherically symmetric, we write
	$G_P:=(\IN,w_P,m_P)$ with
	$m_P(n):=m(S_n(x_0))$ and $w_P(n,n+1) := w(S_n(x_0),S_{n+1}(x_0))$.
\end{defn}
Remark that normalized symmetric means that $G_P$ is normalized and that physically symmetric means that $G_P$ is physical.

\begin{theorem}\label{thm:weaklysphericalsymmetricCD}
	Let $G=(V,w,m)$ be a weakly spherically symmetric graph w.r.t. some $x_0 \in V$. Suppose $G$ is infinite and connected and satisfies $\CCD(K,D)$. Then also $G_P$ satisfies $\CCD(K,D)$.
	\begin{proof}
		Let $f_P : \IN \to \IR$ and let $f(y) := f_P(d(x_0,y))$.
		Observe that \cite[Lemma~3.3]{KellerLenzWojciechowski10} implies for all $x \in V$,
		$$
		\Delta f_P(d(x,x_0)) = \Delta f(x).
		$$
	Since the $\Gamma$-calculus is defined via the Laplace operator $\Delta$, we obtain
	\begin{align*}
	0 &\leq \Gamma_2 f(x) - K\Gamma f(x) - \frac 1 D \Delta f(x)^2 \\&=\Gamma_2 f_P(d(x,x_0)) - K\Gamma f_P(d(x,x_0)) - \frac 1 D \Delta f_P(d(x,x_0))^2
	\end{align*}
where the inequality holds since $G$ satisfies $\CCD(K,D)$ by assumption.		
Hence, $G_P$ satisfies $\CCD(K,D)$ at $d(x,x_0)$ which finishes the proof since $x$ is arbitrary and $d(x,x_0)$ can be arbitrarily large since $G$ is infinite and connected.
	\end{proof}
\end{theorem}

Applying this theorem and Corollary~\ref{cor:CDimpliesVD} immediately yields volume normalized symmetric graphs with non-negative curvature.
\begin{corollary}\label{coro:symm1}
Let $G=(V,w,m)$ be a normalized symmetric graph w.r.t. some $x_0 \in V$. Suppose $G$ is infinite and connected and satisfies $\CCD(0,D)$ for some $D\geq4$. Then, one has $\mathcal{VD}(2^{D-1})$.
\end{corollary}

Applying Theorem~\ref{thm:weaklysphericalsymmetricCD} and Theorem~\ref{thm:scwithdecay} immediately yields stochastic completeness for physically symmetric graphs with curvature not decaying faster than $-R^2$ w.r.t some intrinsic metric.
\begin{corollary}\label{coro:symm2}
Let $G=(V,w,m)$ be a physically symmetric graph w.r.t. some $x_0 \in V$. Let $\rho$ be an intrinsic metric on $G$. Suppose $G$ is infinite and connected and satisfies $\mathcal K_{G,x}(\infty) \geq -C(R^2+1)$ for all $x \in B_R^{\rho}(x_0)$, all $R >0$, and some constant $C$. Then, $G$ is stochastically complete.
\end{corollary}

\subsection{Infinite graphs with positive curvature bounds}\label{s:egg}

Recently it has been shown that there are infinite weighted graphs with a uniformly positive curvature in the sense of Ollivier, see \cite{MuenchWojciechowski17}. 
Naturally, the question arises if this transfers to Bakry-\'Emery curvature.
We give a class of examples of non-Feller graphs satisfying $\CCD(K,D)$ with $K>0$ and $D<\infty$.
To do so, we first give a cutoff lemma.
\begin{lemma}\label{lem:cutoff}
Let  $G=(\IZ,w,m)$ be a linear graph. Let $\widetilde G$ be the restriction of $G$ on $\IN$, i.e., $\widetilde G = (\IN,\widetilde w, \widetilde m)$ with
$\widetilde w = w |_{\IN \times \IN}$ and $\widetilde m = m |_\IN$.
Then, $\mathcal K_{\widetilde G,x}(\mathcal N) \geq \mathcal K_{G,x}(\mathcal N)$ for all $x \in \IN$ and all $\mathcal N >0$.
\end{lemma}
\begin{proof}
This follows from a straight forward case distinction for $x$ and Proposition~\ref{prop:equv}.
\end{proof}

Using a perturbation and a self-similarity argument, we prove the following theorem.

\begin{theorem}\label{thm:infinitePosCurv}
Let $G= G^\omega_\mu :=(\IZ,w,m)$ be given by
\begin{align*}
w(n,n+1) &:= \omega^{-n} \\
m(n) &:= \mu^{-n}
\end{align*}
for some $1< \omega < \mu$.
Then, there is $K>0$ and $D<\infty$ s.t.
$$
\mathcal K_{G,x}(D) \geq K\left(\frac \mu \omega\right)^x.
$$
\end{theorem}

\begin{proof}
We write $\alpha := \mu/\omega$. Observe $d_+(n) = \alpha^n$ and $d_-(n)=\omega \alpha^n$.
Hence,
\begin{align*}
W_-(n) &= \frac12\left( -d_-(n-1) + 3 d_+(n-1) + \Big( 1- \frac 4 D \Big)d _-(n)-d_+(n)- 2K(n)\right)   \\
&=\frac 1 2 \left(- \omega \alpha^{n-1} + 3 \alpha^{n-1} + \left( 1 - \frac 4 D \right)\omega \alpha^n - \alpha^n - 2K \alpha^n \right)\\
&=\frac 1 2 \alpha^n \left(- \frac \omega \alpha +\frac  3 \alpha + \left( 1 - \frac 4 D \right)\omega  - 1 - 2K \right)\\
&=:\frac 1 2 \alpha^n  F_-(\alpha,\omega,D,K)
\end{align*}
and analogously,
\begin{align*}
W_+(n) &= \frac12\left( -d_+(n+1) + 3 d_-(n+1) + \Big( 1- \frac 4 D\Big)d _+(n)-d_-(n)- 2K(n)\right) \\
&=\frac 1 2 \alpha^n \left(-\alpha + 3 \omega \alpha + \Big( 1- \frac 4 D\Big) - \omega - 2K  \right)\\
&=:\frac 1 2 \alpha^n  F_+(\alpha,\omega,D,K).
\end{align*}
It is easy to see that $F_\pm(\alpha,\omega,D,K)>0$ if $\alpha>1$ and $\omega >1$ and if $D$ large enough and $K>0$ small enough.
Hence due to Proposition~\ref{prop:equv}, $\CCD(K(n),D,n)$ is satisfied if and only if
\begin{align*}
\frac 1 4 \alpha^{2n} F_-(\alpha,\omega,D,K)F_+(\alpha,\omega,D,K) &\geq \left( 1 - \frac 2 D \right)^2 d_+(n)d_-(n)\\
&=  \left( 1 - \frac 2 D \right)^2 \omega \alpha^{2n}.
\end{align*}

We  now show that
$$
\frac 1 4 F_-(\alpha,\omega,\infty,0)F_+(\alpha,\omega,\infty,0) > \omega
$$
which would finish the proof due to continuity of all terms in $K$ around zero and in  $D$ around infinity.
We have
\begin{align*}
F_-(\alpha,\omega,\infty,0)F_+(\alpha,\omega,\infty,0) - 4 \omega 
&=\left(-\frac \omega \alpha + \frac  3 \alpha + \omega - 1\right)\left(-\alpha + 3\omega\alpha + 1 - \omega\right) - 4\omega\\
&=\frac{(\alpha - 1)(\omega - 1)(3 - \omega + \alpha\left( 3 \omega - 1\right))}{\alpha} \\
&>0
\end{align*}
since $\alpha > 1 $ and $\omega>1$.  This finishes the proof.
\end{proof}

Before stating the next result, we remind the reader of the Feller property on graphs.
A graph $G=(V,w,m)$ is Feller, if $e^{t\Delta}: C_0(V) \to C_0(V)$ where $C_0(V)$ is the closure of finitely supported functions under the supremum norm.
The following corollary is obtained by
restricting the graph from Theorem~\ref{thm:infinitePosCurv} to $\IN$ which does not decrease the curvature due to Lemma~\ref{lem:cutoff}. 

\begin{corollary}
There exist $K>0$ and $D<\infty$ and non-Feller and non-complete graphs satisfying $\CCD(K,D)$.
\end{corollary}

\begin{proof}
Let $G$ be the restriction to $\IN$ of $G_\mu^\omega$ from Theorem~\ref{thm:infinitePosCurv} with some $\mu>\omega>1$.
Combining Theorem~\ref{thm:infinitePosCurv} and Lemma~\ref{lem:cutoff} yields $\CCD(K,D)$ for some $K>0$ and $D<\infty$. 

We now show incompleteness by contradiction.
Suppose $G$ is complete.
Let $\eta \in C_c(V)$ s.t. $\eta(0)=1$ and $\Gamma \eta \leq \eps$.
Then, $\eta(r) - \eta(r+1) \leq \sqrt{2\eps}/\sqrt{d_+(r)}=\sqrt{2\eps}(\omega/\mu)^{r/2}.$
Therefore,
\[1=\eta(0)= \sum_{r\geq 0} \eta(r) - \eta(r+1) \leq \frac{\sqrt{2\eps}}{1 - \sqrt{\omega/\mu}}<1\] if $\eps$ small which is a contradiction, and thus, proves incompleteness.

It is left to show that the graph is not Feller.
Since $G$ is spherically symmetric, by \cite[Theorem~4.13]{wojciechowski2017feller} it suffices to show
$m(V) < \infty$ and
$$
\sum_n \frac {\sum_{k>n} m(k)}{w(n,n+1)} < \infty.
$$
But this follows since
$$
\sum_{k>n} m(k) = \sum_{k>n} \mu^{-k} = \frac{\mu^{-n}}{\mu - 1}
$$
and therefore,
\begin{align*}
\sum_n \frac {\sum_{k>n} m(k)}{w(n,n+1)} = \sum_n \frac{\mu^{-n}}{(\mu - 1)\omega^{-n}} < \infty
\end{align*}
since $\mu>\omega$. This finishes the proof.
\end{proof}

Remark that these graphs even have exponential growth of the curvature.

\bigskip
\textbf{Acknowledgements:} B. H. is supported by NSFC, grant no. 11401106. F. M. wants to thank the German National Merit Foundation for financial support.

\TOCstop

\bibliography{Bibliography}
\bibliographystyle{alpha}

Bobo Hua,\\
School of Mathematical Sciences, LMNS, Fudan University, 200433 Shanghai, China; \\
\texttt{bobohua@fudan.edu.cn}
\\

Florentin M\"unch, \\
Department of Mathematics,
University of Potsdam, Potsdam, Germany,\\
Currently:
Center of Mathematical Sciences and Applications, Harvard University, Cambridge MA, USA\\
\texttt{chmuench@uni-potsdam.de}\\

\TOCstart

\end{document}